\documentclass[11pt]{article}
\usepackage[margin=1.24in]{geometry}
\usepackage[english]{babel} 
\usepackage{amsmath, amssymb, amsthm, mathtools, dsfont} 
\usepackage[normalem]{ulem}
\usepackage[dvipsnames]{xcolor} 
\usepackage{graphicx, booktabs} 
\usepackage[bookmarks=false]{hyperref} 
\usepackage{comment, autonum} 
\usepackage{enumerate}
\definecolor{darkgreen}{rgb}{0,0.5,0}

\newcommand{\Honetwo}{{H_\diamondsuit^2(\Omega)}}

\newcommand{\Honethree}{{H_\diamondsuit^3(\Omega)}}
\newcommand{\Honefour}{{H_\diamondsuit^4(\Omega)}}
\newcommand{\Honefive}{{H_\diamondsuit^5(\Omega)}}
\DeclarePairedDelimiter{\paren}{\lparen}{\rparen}
\DeclarePairedDelimiter{\abs}{\lvert}{\rvert}

\newcommand{\tild}[1]{\tilde{#1}} 
\newcommand{\R}{\mathbb{R}} 
\newcommand{\N}{\mathbb{N}} 
\newcommand{\bigO}{\mathcal{O}}
\newcommand{\e}{\mathrm{e}} 
\newcommand{\id}{I} 
\DeclareMathOperator{\diagmatrix}{diag} 
\DeclareMathOperator{\tr}{tr} 
\newcommand{\transpose}{\top} 
\newcommand{\ccdot}{\mathbin{:}}
\newcommand{\boundary}{\partial}

\newcommand{\embed}{\hookrightarrow}

\newcommand{\conv}{\mathbin{\ast}}

\newcommand{\E}{\mathcal{E}}
\newcommand{\dx}[1]{\partial_{x_{#1}}}
\newcommand{\dt}{\partial_{t}}
\newcommand{\ddt}{\frac{\textup d}{\textup{d} t}}
\newcommand{\dxx}{\,\textup{d}{x}}
\newcommand{\ds}{\,\textup{d}{s}}
\newcommand{\dS}{\,\textup{d}{S}}
\newcommand{\jac}{\boldsymbol{\nabla}} 
\newcommand{\lapl}{\Delta} 
\newcommand{\grad}{\nabla} 
\newcommand{\divv}{\nabla \cdot} 
\newcommand{\curl}{\nabla \times} 
\newcommand{\Divv}{\boldsymbol{\nabla} \cdot } 
\newcommand{\laplvec}{\Delta}
\newcommand{\quater}{\frac{1}{4}}

\newcommand{\half}{\frac{1}{2}}
\newcommand{\intt}{\int_0^t}
\newcommand{\den}{\varrho}
\newcommand{\vel}{v}
\newcommand{\pre}{p}
\newcommand{\ent}{s}
\newcommand{\tem}{T}
\newcommand{\stress}{\sigma}
\newcommand{\strain}{\varepsilon}
\newcommand{\force}{f_B}
\newcommand{\heat}{q}
\newcommand{\work}{\phi}
\newcommand{\heatK}{\mathcal{K}}
\newcommand{\shearK}{\mathcal{S}}
\newcommand{\bulkK}{\mathcal{B}}
\newcommand{\viscK}{\mathcal{V}}
\newcommand{\mach}{M}
\newcommand{\de}{\tild{\den}}
\newcommand{\ve}{\tild{\vel}}
\newcommand{\pr}{\tild{\pre}}
\newcommand{\en}{\tild{\ent}}
\newcommand{\te}{\tild{T}}

\newcommand{\calT}{\mathcal{T}}
\newcommand{\calB}{\mathcal{B}}
\newcommand{\calU}{\mathcal{U}}
\newcommand{\lin}{M}
\newcommand{\dam}{J_{0}}
\newcommand{\laplu}{v}

\newcommand{\Ltwo}{L^2(\Omega)}
\newcommand{\Linf}{L^\infty(\Omega)}

\newtheorem{theorem}{Theorem}
\newtheorem{lemma}[theorem]{Lemma}

\newtheorem*{assumption*}{Assumptions}
\newtheorem{assumption}{Assumption}

\newtheorem{remark}{Remark}
\numberwithin{theorem}{section}
\numberwithin{equation}{section}
\makeatletter
\newcommand{\leqnomode}{\tagsleft@true}
\newcommand{\reqnomode}{\tagsleft@false}
\makeatother

\title{Well-posedness of a first-order formulation for fractionally damped nonlinear acoustics}

\author{Pascal Lehner\footnote{Department of Mathematics, University of Klagenfurt, Klagenfurt, 9020, Carinthia, Austria. E-mail: \href{mailto:Pascal.Lehner@aau.at}{Pascal.Lehner@aau.at}}, Mostafa Meliani\footnotemark[2]\footnote{Department of Mathematical Sciences, University of Bath, Bath, UK. E-mail: \href{mailto:mm4138@bath.ac.uk}{mm4138@bath.ac.uk}}
}
\date{}
\begin{document}
\maketitle
\begin{abstract}
In this work, we study the well-posedness of a quasilinear first-order-in-time system with memory arising in nonlinear acoustics. The model features a memory kernel describing the fractional damping and covers, as special cases, first-order formulations of Kuznetsov- and Westervelt-type equations. 

For completely monotone convolution kernels, we prove that the Westervelt-type system admits unique local-in-time solutions in bounded domains for space dimensions $d\leq3$, under homogeneous Dirichlet boundary conditions, suitable regularity, and smallness assumptions. The analysis is based on energy estimates exploiting novel nonlinear coercivity properties of the memory terms. Under additional assumptions on the resolvent of the kernel, we show that the smallness conditions can be significantly relaxed. For finite sums of exponential kernels our results generalize to the Kuznetsov-type system. 

Additionally, we show that the inviscid case (absence of memory kernel) in \(\R^d\) can be treated by standard hyperbolic arguments in the case of the Kuznetsov-type system.

\end{abstract}
\section{Introduction}
We study an initial boundary value problem for a quasilinear hyperbolic system with damping memory effects and establish well-posedness for special cases of the system modeling fractionally damped nonlinear acoustic waves. The general system of interest in this paper is
\begin{equation}\label{eq:main_system}
\left\{\begin{array}{ll}
A_0 u_t + \sum_{j=1}^d A_j(u,x) u_{x_j}  - \mathcal{J} \conv \Delta u = g &\mathrm{ on } \qquad (0,T) \times   \Omega,
\\
u(0) = u_0 &\mathrm{ on } \qquad \Omega,
\\
u = 0 &\mathrm{ on } \qquad (0,T) \times \partial \Omega,
\end{array}
\right.
\end{equation}
posed on the space–time cylinder $(0,T)\times\Omega$, where $T>0$ and $\Omega\subseteq \mathbb{R}^d$ is an open, sufficiently smooth domain. The unknown is $u=(p,v)^\transpose\in\R^{1+d}$, with scalar 
{pressure} $p$ and  
{velocity} vector
$v\in\R^d$. The matrix $A_0\in\R^{(1+d)\times(1+d)}$ is symmetric positive definite and independent of space and time, and the flux is encoded by
\begin{equation}\label{eq:Aj_matrices}
A_j(u,x) = A_j(p,v, x) \;=\; \begin{pmatrix} \vartheta v_j & (1+ \alpha(x) + \lambda p)e_j^\top \\ (1 + \beta(x) + \mu   p)e_j & \kappa e_j v^\top  \end{pmatrix} \in \R^{(1+d) \times (1+d)}, 
\end{equation}
where $e_j\in\R^d$ 
are the canonical $j$-th unit vectors for $1\leq j \leq d$. The parameters $\kappa,\lambda,\mu,\vartheta\in\mathbb{R}$ with \( \lambda \neq \vartheta \) govern the nonlinear couplings, while $\alpha,\beta$ are space-dependent coefficients. A physical choice of the parameters is described in Appendix \ref{se:derivation} in \eqref{eq:physical_parameters}.  The source $g:(0,T) \times \Omega \to \R^{1+d}$ and the initial datum $u_0: \Omega \to \R^{1+d}$ are given. The operator $\conv$ denotes convolution in time, i.e., $(\mathcal{J}\conv \Delta u)(t)=\int_0^t \mathcal{J}(t-s)\,\Delta u(s)\,ds$, with $\mathcal{J}\in L^1(0,T)^{1+d}$ a given memory kernel acting componentwise. 
For simplicity, we impose homogeneous Dirichlet boundary conditions $u=0$ on $(0,T)\times\partial\Omega$.

Observe that the matrix in \eqref{eq:Aj_matrices} is generally not symmetric, a property typically required for proving well-posedness of first-order hyperbolic systems, see \cite{Kato1975}. We overcome this difficulty by symmetrizing the system \eqref{eq:main_system} using a suitably chosen multiplicative matrix $S$; see Section~\ref{se:well-posedness_Kuznetsov} for details.
This symmetrization is crucial for our results when \( \mathcal{J} \in L^1(0,T)^{1+d} \).
In contrast, when $\mathcal{J}=\delta_0$ (with $\delta_0$ denoting the Dirac measure at $t=0$) well-posedness of the general system \eqref{eq:main_system} is shown in \cite{QuadraticWave} and relies on the parabolicity of the system exploiting regularity from the Laplacian \( - \Delta \).  For a general kernel $\mathcal{J} \in L^1(0,T)^{1+d}$, the difficulties in establishing the well-posedness of \eqref{eq:main_system} stem from the fact that less regularity can be inferred from the non-local damping term. More details are given in Section \ref{se:well-posedness}.

Within the above framework, we establish two types of well-posedness results.
\begin{itemize}
    \item Theorems \ref{th:nonlin-well} and \ref{th:nonlin-well-refined}, valid for the class of fluxes with $\mu=\kappa=0$ and for completely monotone kernels $\mathcal{J}$, state that the problem \eqref{eq:main_system} is well posed in a bounded domain for $d \leq 3$ under natural regularity assumptions together with a smallness condition on the data. This is the main part of the paper and the proofs use novel energy estimates (see Theorem \ref{thm:lin-well}) and nonlinear coercivity estimates for $\mathcal{J}$ (see Lemmas~\ref{lemma:kernel}, \ref{lemma:kernel2}).
    \item Theorem \ref{th:unbounded_domain} covers the inviscid case $\mathcal{J}=0$ with $\vartheta=\kappa$ on the whole space $\Omega=\R^d$ for arbitrary dimensions \( d \). It establishes local-in-time well-posedness relying on standard hyperbolic techniques developed in \cite{Kato1975}. 
\end{itemize}



The general system \eqref{eq:main_system} is linked to several classical models in (nonlinear) acoustics. Setting $\alpha=\beta=0$ yields a first-order-in-time formulation of Kuznetsov’s equation; see \cite[Rem.~2.4]{QuadraticWave}. If, in addition, $\mu=\kappa=0$, the system reduces to a Westervelt-type equation in first-order-in-time form, see Appendix \ref{se:derivation} for details. Moreover, taking $\lambda=\vartheta=0$ gives a linear fractionally damped wave equation. Finally, setting $\mathcal{J}=0$ leads to the first-order-in-time formulation of linear, lossless acoustic wave propagation
\begin{equation}\label{eq:main_system_linear2}
\left\{
\begin{aligned}
p_t + \divv v &= g_p,\\
 v_t + \grad p &= g_v.
\end{aligned}
\right.
\end{equation}

{In many applications, the acoustic damping and nonlinear coefficients are typically small but generally not negligible; see, e.g.,~\cite{Hamilton:1997} for a review of nonlinear acoustic wave propagation and the physical relevance of strongly damped nonlinear models such as Kuznetsov's or Westerwelt's equations. 
}

\subsection{Motivation of the model}
Classical nonlinear acoustics models are typically derived from the Navier–Stokes equations, complemented with constitutive relations such as a Newtonian fluid stress tensor for the acoustic medium and Fourier’s law of heat conduction; see, e.g., \cite{Hamilton:1997, Pierce:2019}. However, these assumptions imply infinite propagation speed and predict viscous losses that increase quadratically with frequency, which is inconsistent with observations in medical ultrasound \cite[\S 5.1]{Holm2019}. To address these shortcomings, two approaches have been developed. The first, introduced by Gurtin and Pipkin~\cite{GurtinPipkin1968}, modifies the heat flux, while the second incorporates power laws into the viscoelastic stress tensor; see, e.g., \cite{Holm2019}. 
These approaches lead to second-order nonlinear acoustic models featuring fractional damping terms as derived in, e.g., \cite{Holm2013,Prieur2011}.

Motivated by the suitability of first-order systems for the development of efficient numerical schemes, we consider the recently proposed model 
\[ 
\label{eq:first_order}
\left\{
\begin{aligned}
p_t + (1 + \alpha + \lambda p) \divv v - \mathcal{K} * \lapl p  + \grad p \cdot v &= 0, \\
v_t + (1 + \beta) \grad p - \mathcal{V} * \lapl v +  \half k \grad(|v|^2 - p^2) &= 0,
\end{aligned}
\right.
\]
with \( k \in \{0,1\} \),
which is derived and analyzed in \cite{QuadraticWave} for the case $\mathcal{K}=\mathcal{V}=\delta_0$. The system \eqref{eq:first_order} is in dimensionless form and can be obtained from the general system \eqref{eq:main_system} by setting $\kappa=\mu=k$ and $\vartheta=1$.  
A derivation of \eqref{eq:first_order} with \( \mathcal{K}, \mathcal{V} \neq \delta_0\) is provided in Appendix~\ref{se:derivation}.  
The model incorporates fractional damping through stress-strain relations and heat flux laws using separate memory kernels in the mathematical formulation. The kernel $\mathcal{K}$ corresponds to the heat conduction law, while $\mathcal{V}$ represents a generalized stress-strain relation modeling viscoelasticity.

\subsection{Related mathematical literature}
The well-posedness analysis of the fractionally damped Westervelt equation in pressure form (which can be derived from \eqref{eq:main_system} by setting $\alpha=\beta=\mu=\kappa=0$, see \cite[Rem.~2.4]{QuadraticWave}) 
\[ 
\label{eq:fractional_westervelt}
p_{tt} - \Delta p - \mathcal{K} * \Delta p_t = \lambda (p^2)_{tt}
\]
has only recently been established; see \cite{BakerBanjaiPtashnyk2024, KaltenbacherRundell2022} for Abel kernels (defined in \eqref{eq:frac_kernel}) and \cite{kaltenbacher2024limiting} for more general kernels. 
Here, $\mathcal{K}$ is a suitable time-dependent memory kernel, and $\lambda>0$ is a nonlinearity parameter. In \cite{kaltenbacher2024limiting}, a limiting analysis as $\mathcal{K} \to 0$ is performed, recovering the lossless Westervelt model. Fractional versions of Kuznetsov and Blackstock were studied in \cite{KaltenbacherMelianiNikolic2024}. 

Another related model is the fractional Jordan–Moore–Gibson–Thompson equation, arising from incorporating thermally relaxed heat flux laws of Maxwell--Cattaneo and Compte--Metzler type~\cite{compte1997, MaxwellCattaneo}. These models were first derived in \cite{KaltenbacherNikolic2022} and their well-posedness and singular-limit analyses have been carried out in~\cite{frac_tau2zero_PartII,KaltenbacherNikolic2022, meliani2023unified}. A global existence result has been recently established in \cite{meliani2025global} exploiting the concept of strongly positive kernels similarly to \cite{cannarsa2011integro,Okada2021}. 

In the broader acoustics and viscoelastic wave literature, much work has been done in the study of hyperbolic (second-order and third-order) equations with smooth integrable kernels; 
we refer the reader to, e.g., \cite{conti2023moore,Dharmawardane2011,lasiecka2017global}.
In contrast, the present work relies on complete monotonicity assumptions (see Assumptions~\ref{as:kernel}--\ref{as:kernel2} below), enabling us to cover large families of relevant kernels for applications; see Section~\ref{sec:coverd_kernels}.

While the aforementioned well-posedness results for fractionally damped nonlinear acoustics models concern second-order-in-time (e.g., Westervelt or Kuznetsov) and third-order-in-time (Jordan–Moore–Gibson–Thompson) equations, fewer results exist for first-order formulations.
To the best of our knowledge, the only such result is due to \cite{cox:2024}, where well-posedness is established for a space-fractionally damped first-order nonlinear acoustic model.
Our work seeks to fill this gap in the analysis for first-order time-fractional models.

\subsection{Examples of relevant kernels}\label{sec:coverd_kernels}
Although the Abel kernel of order $\alpha \in (0,1)$,
\begin{equation}\label{eq:frac_kernel}
\mathcal{J}: (0, \infty) \to \R, \quad \mathcal{J}(t) = \frac1{\Gamma(\alpha)} t^{\alpha-1},
\end{equation}
is our main example of interest as it is used in the definition of Caputo--Djrabashyan fractional derivatives, see \cite{podlubny1998fractional}, the assumptions we impose on the kernels  (see Assumptions~\ref{as:kernel}--\ref{as:kernel2} below) are sufficiently general to include a broad class of kernels. In particular, they cover kernels arising in the modeling of acoustic and viscoelastic phenomena such as
\begin{itemize}
\item the exponential kernel 
\begin{equation}\label{Exponential_Kernel}
\mathcal{J}(t) = \beta \mathrm{e}^{-\beta t}, \qquad \beta>0.
\end{equation}
\item the exponentially regularized Abel kernel, along the lines of that found in~\cite{messaoudi2007global},
\begin{equation}\label{eq:reg_abel_kernel}
\mathcal{J}(t) =\frac{t^{\alpha-1}\mathrm{e}^{-\beta t}}{\Gamma(\alpha)} , \qquad  0<\alpha < 1, \, \beta > 0. 
\end{equation}
\item the fractional Mittag-Leffler kernels, encountered in the study of fractional second order wave equations in complex media in~\cite{kaltenbacher2024limiting},
\begin{equation}\label{eq:mittag_leffler}
\mathcal{J}(t) =\frac{t^{\beta-1}}{\Gamma(1-\alpha)}E_{\alpha,\, \beta}\left(-t^\alpha\right), \qquad   0< \alpha \, { \leq }\, \beta < 1, 
\end{equation} 
where $E_{\alpha, \, \beta}(t)=\sum_{k=0}^\infty \frac{t^k}{\Gamma(\alpha k + \beta)}$ is the two-parametric Mittag-Leffler function. 
\item the polynomially decaying kernel 
\begin{equation}\label{Polynomial_Kernel}
\mathcal{J}(t) = \frac{1}{(1+t)^{q}}, \qquad q>0
\end{equation}
arising in viscoelasticity, see, e.g.,~\cite{munoz1996decay}.
\end{itemize}
Here, $\Gamma: (0, \infty) \to \R, \, t \mapsto \int_0^\infty s^{t-1} \e^{-s} \, \ds$ stands for the Gamma function. Note that all kernels should be scaled in a physically meaningful way.

\subsection{Main results}
We state our well-posedness results under the hypothesis that the vector-valued memory kernel $\mathcal{J}$ is completely monotone on $(0,T)$ in the sense that
\begin{equation}\label{def::complete_monot}
(-1)^n \frac{d^n}{dt^n}\mathcal{J}(t) \ge 0 \qquad \text{for all } n \in \mathbb{N}_0 \text{ and a.e.\ } t \in (0,T),
\end{equation}
where the inequality is understood componentwise and derivatives may be taken in the distributional sense. {Informally, complete monotonicity means that $\mathcal{J}$ is non-negative and non-increasing, decaying rapidly at short times and more slowly at large times.}

We formulate two nested sets of assumptions, leading to two distinct well-posedness theorems. Under the first (weaker) assumption, well-posedness is obtained at the expense of stronger smallness conditions on the data. The second (stronger) assumption strictly refines the first and allows us to relax these smallness requirements. The assumptions are as follows.
\begin{assumption}\label{as:kernel}
The kernel \(\mathcal{J} : (0, T) \to \R^{1+d} \) in \eqref{eq:main_system} lies in \(L^1(0,T)^{1+d}\), is completely monotone, and {non-constant} or \( \mathcal{J} = 0 \).
\end{assumption}
\begin{assumption}\label{as:kernel2}
Let Assumption~\ref{as:kernel} hold.  In addition if \( \mathcal{J} \neq 0\), suppose that there exist $T>0$ and $q>1$ such that the resolvent of first kind of $\mathcal{J}$ defined by 
\[
(\mathfrak{r} * \mathcal{J})(t)
= 1
\qquad \text{for a.e.\ } t \in (0,T)
\]
lies in $L^q(0,T)^{1+d}$. 
\end{assumption}
We note here that the resolvent $\mathfrak{r}$ of a completely monotone kernel is in general a Radon measure, 
see e.g., \cite[Ch. 5]{Gripenberg1990} and the discussion in Appendix~\ref{se:kernel}. The assumption on the regularity of the resolvent associated with the kernel ensures that $\mathcal{J}$ possesses a sufficiently strong (yet integrable) singularity at $t=0$. 

Note that the before mentioned kernels \eqref{eq:frac_kernel} to \eqref{Polynomial_Kernel} and \(\mathcal{J}=0\) all fulfill Assumption \ref{as:kernel} within the given parameter range.  For further details, see \cite[Cor. 3.2]{jin2021fractional} concerning the two parametric Mittag-Leffler functions. Assumption \ref{as:kernel2} is satisfied for the (regularized) Abel and Mittag-Leffler kernels \eqref{eq:frac_kernel}, \eqref{eq:reg_abel_kernel}, and \eqref{eq:mittag_leffler}. It does, however, not hold for the exponential \eqref{Exponential_Kernel} or polynomially decaying \eqref{Polynomial_Kernel} kernels, since their resolvents contain a Dirac pulse \(\delta
_0\) due to the absence of singularity at $0$; see \cite[Ch. 5, Thm. 5.4]{Gripenberg1990}.

For the well-posedness analysis, we rely on nonlinear coercivity of $\mathcal{J}\in L^1(0,T)^{1+d}$ in the following sense. We denote as $\lambda_{\min}$ the smallest eigenvalue of a matrix. 
\begin{lemma} \label{lemma:kernel}
Let $V \in W^{1, \infty}\! \left(0,T;L^\infty(\Omega;\R^{(1+d)\times(1+d)})\right)$ be a diagonal and positive definite matrix. Assume that there exists a constant $C_K>0$, possibly depending on $\mathcal{J},T$, such that
\[ \label{eq:kerlem_assump}
\|\dt V\|_{L^\infty(0,T;L^\infty(\Omega;\R^{(1+d)\times(1+d)}))}  < C_K \inf_{ {s \in (0,T), \, x \in \Omega}} \lambda_{min}(V)(s,x) .
\]
If the kernel $\mathcal{J}$ satisfies Assumption \ref{as:kernel}, then for every $y \in L^2 \! \left(0,T;L^2(\Omega;\R^{1+d})\right)$ and a.e. $t \in (0,T)$ the estimate
\begin{equation}\label{eq:kernel_assump}
\begin{aligned}
\int_0^t ( V(s) \big( \mathcal{J} * y \big)(s) , y(s) )_{L^2(\Omega;\R^{1+d})} \ds  \geq C_{\mathcal{J}} \int_0^t  \| \big( \mathcal{J} * y \big)(s) \|^2_{L^2(\Omega;\R^{1+d})}  \ds 
\end{aligned}
\end{equation}
holds, where $C_{\mathcal{J}}>0$ depends on $\Omega, T, V,$ and $\mathcal{J}$. 
\end{lemma}
\begin{proof}
See Appendix~\ref{sec:proof34}.
\end{proof}
We note that for constant matrices \(V\), the (linear) coercivity estimate is shown for completely monotone kernels in \cite{kaltenbacher2024limiting}. 

Using the more restrictive Assumption~\ref{as:kernel2}, the estimate \eqref{eq:kernel_assump} can be obtained without requiring a smallness condition on the time derivative of $V$ relative to the smallest eigenvalue of $V$, see \eqref{eq:kerlem_assump}, provided that the time horizon $T>0$ is chosen sufficiently small.
\begin{lemma} \label{lemma:kernel2}
{Let $V \in W^{1, \infty}\! \left(0,T;L^\infty(\Omega;\R^{(1+d)\times(1+d)})\right)$ be a diagonal and positive definite matrix and let the kernel $\mathcal{J}$ satisfy Assumption \ref{as:kernel2}}, then there exists a $T'>0$ such that for every $y \in L^2 \! \left(0,T';L^2(\Omega;\R^{1+d})\right)$ and a.e. $t \in (0,T')$ the estimate \eqref{eq:kernel_assump} holds. 
\end{lemma}
\begin{proof}
    See Appendix~\ref{sec:proof35}.
\end{proof}
\begin{remark}
The assumption that the matrix \(V\) is diagonal is essential in the proof of Lemma \ref{lemma:kernel}. In particular, this ultimately restricts our well-posedness result to the Westervelt-type system, see Theorems \ref{th:nonlin-well}, \ref{th:nonlin-well-refined}. A discussion of this requirement is provided in Appendix \ref{se:appendix_lemma_assumptions}, where we prove Lemma \ref{lemma:kernel} for non-diagonal matrices in the case of an exponential kernel; see Appendix~ \ref{se:appendix_kernel_exponential_proof}. 
\end{remark}

To state the main results, we introduce the following shorthand notation of the relevant function spaces
\[ 
\begin{aligned}
H_\diamondsuit^2(\Omega;\R^m) &= H^2(\Omega;\R^m) \cap H^1_0(\Omega;\R^m), \\
H_\diamondsuit^3(\Omega;\R^m)  &=  \left\{u \in H^3(\Omega;\R^m) \cap H_0^1(\Omega;\R^m)\ : \ \Delta u|_{\partial\Omega} = 0\right\}, \\
H_\diamondsuit^4(\Omega;\R^m) &= H^4(\Omega;\R^m) \cap H_\diamondsuit^3(\Omega;\R^m), \\
H_\diamondsuit^5(\Omega;\R^m)  &=  \left\{u \in H^5(\Omega;\R^m) \cap H_\diamondsuit^3(\Omega;\R^m) \ : \ \Delta^2 u|_{\partial\Omega}  = 0\right\}.
\end{aligned}
\]
We define the solution space by
\begin{equation}
\mathcal{U}_{\diamondsuit}\! \left(0,T\right) = L^\infty \! \left(0,T;H_\diamondsuit^4(\Omega;\R^{1+d})   \right) \, \cap\,  H^1 \! \left(0,T; {H_\diamondsuit^3(\Omega;\R^{1+d}) }\right) {\, \cap\, W^{1, \infty} \! \left(0,T;H_\diamondsuit^2(\Omega;\R^{1+d}) \right)} 
\end{equation}
and the analog space without boundary conditions
\begin{equation}
\mathcal{U}\! \left(0,T\right) = L^\infty \! \left(0,T; H^4(\Omega;\R^{1+d})   \right) \, \cap\,  H^1 \! \left(0,T; H^3(\Omega;\R^{1+d})\right) {\, \cap\, W^{1, \infty} \! \left(0,T;H^2(\Omega;\R^{1+d})\right)} 
.
\end{equation}
\begin{theorem}[General Westervelt-type case]\label{th:nonlin-well}
Let $\mu=\kappa=0$ in \eqref{eq:Aj_matrices}. Let $T>0$ and let {$\Omega\subset \R^d$}, $d \leq 3$, be a bounded domain with boundary of class $C^{4,1}$. Assume $\mathcal{J} \in L^1(0,T)^{1+d}$ satisfies Assumption \ref{as:kernel} and
\[
u_0 \in {H_\diamondsuit^4(\Omega;\R^{1+d})}, \quad {g \in \mathcal{V}(0,T) := L^1 \! \left(0,T;H_\diamondsuit^4(\Omega;\R^{1+d})\right)
\cap L^{\infty}\! \left(0,T;H_\diamondsuit^2(\Omega;\R^{1+d}) \right)},
\]
and $\alpha, \beta \in H^{4}(\Omega)$ are as in~\eqref{eq:Aj_matrices} with $\| \alpha \|_{L^\infty(\Omega)}, \| \beta \|_{L^\infty(\Omega)}<1$. Assume there exists $r_0>0$ sufficiently small such that 
\[ \label{eq:smallness_1}
\| u_0 \|_{H_\diamondsuit^4(\Omega;\R^{1+d})} + \| g \|_{\mathcal{V}(0,T)} \leq r_0 .
\]
Then there exists a $T'>0$ such that equation~\eqref{eq:main_system} has a unique solution 
$u \in \mathcal{U}_\diamondsuit(0,T')$ with \emph{a priori} estimate 
\[ \label{eq:apriori_nonlinear}
\| u \|_{\calU(0,T')} \leq C \left( \| u_0 \|_{H_\diamondsuit^4(\Omega;\R^{1+d})} + \| g \|_{\mathcal{V}(0,T)} \right)
\]
for a constant $C>0$ that depends only on $A_j, A_0, \mathcal{J}, T, \Omega, d$.
\end{theorem}
For the next result, we require greater regularity of the resolvent of $\mathcal{J}$, which allows us to relax the smallness conditions of Theorem~\ref{th:nonlin-well}. Heuristically, more regular resolvents are associated with more singular kernels, as can be seen from the limiting case that the resolvent of the Dirac delta $\delta_0$ is the constant function 1.

In particular, we show that, as long as the initial pressure $p_0$ is below an explicit threshold defined by the nonlinearity $\lambda$, the system can be evolved for a short time.
\begin{theorem}[More singular kernel Westervelt-type case]\label{th:nonlin-well-refined}
Let $\mu=\kappa=0$ in \eqref{eq:Aj_matrices}. Let $T>0$ and let {$\Omega\subset \R^d$}, $d \leq 3$, be a bounded domain with boundary of class $C^{4,1}$. Assume $\mathcal{J} \in L^1(0,T)^{1+d}$ satisfies Assumption \ref{as:kernel2} and
\[
u_0=(p_0,v_0) \in {H_\diamondsuit^4(\Omega;\R^{1+d})}, \quad {g \in \mathcal{V}(0,T) := L^1 \! \left(0,T;H_\diamondsuit^4(\Omega;\R^{1+d})\right)
\cap L^{\infty}\! \left(0,T;H_\diamondsuit^2(\Omega;\R^{1+d}) \right)},
\]
and $\alpha, \beta \in H^{4}(\Omega)$ are as in~\eqref{eq:Aj_matrices} with $\| \alpha \|_{L^\infty(\Omega)}, \| \beta \|_{L^\infty(\Omega)}<1$, assume 
\[ \label{eq:refine_theorem_assumption}
\|p_0 \|_{L^\infty(\Omega;\R^{1+d})} < \frac{1}{2\lambda } (1- \| \alpha \|_{\Linf}) .
\]
Then there exists a $T'>0$ such that equation~\eqref{eq:main_system} has a unique solution 
$u \in \mathcal{U}_\diamondsuit(0,T')$ with \emph{a priori} estimate 
\[ \label{eq:apriori_nonlinear}
\| u \|_{\calU(0,T')} \leq C \left( \| u_0 \|_{H_\diamondsuit^4(\Omega;\R^{1+d})} + \| g \|_{\mathcal{V}(0,T)} \right)
\]
for a constant $C>0$ that depends only on $A_j, A_0, \mathcal{J}, T, \Omega, d$.
\end{theorem}
\begin{remark}
Strictly speaking, our results are applicable not only to integrable kernels, but also to distributional kernels of the form \(\mathcal{J} = J + A \delta_0\) with \( {J} \in L^1(0,T)^{1+d}\) and \(A\) positive definite. However, we do not cover this case, since if \(A \neq 0\), parabolic estimates are available, and well-posedness can be shown as in \cite{QuadraticWave}.

\end{remark}
Finally, when $\mathcal{J}=0$ we obtain a classical quasilinear hyperbolic well-posedness result stated on \(\Omega=\R^d\) without the assumption \( \mu = \kappa = 0\).
\begin{theorem}[Inviscid Kuznetsov-type case] \label{th:unbounded_domain}
Let $T>0$, $\Omega=\R^d$, $\vartheta=\kappa$ in~\eqref{eq:Aj_matrices} and $\mathcal{J} = 0$. For some $s > 1 + \frac{d}{2}$, assume
\begin{equation}
\begin{aligned}
u_0 \in U_0 &:= \left \{   u \in H^s(\R^d;\R^{1+d}) : \| u \|_{H^s} < C_{L^\infty \embed H^s}^{-1} r_2  \right \},
\\
\| \alpha \|_{L^\infty(\R^d)} &< r_\alpha, \quad \| \beta \|_{L^\infty(\R^d)} < r_\beta, \quad
g \in C([0,T];H^s(\R^d;\R^{1+d}))
\end{aligned}
\end{equation}
with $r_\alpha<\frac{1}{16}$, $r_\beta<\frac{1}{16}$, $r_2 =  \frac{1}{16}\min\{ |\lambda|^{-1}, |\mu|^{-1}, |\kappa|^{-1}  \}$.
Then, there exists a $T'>0$ such that equation~\eqref{eq:main_system} has a unique solution
$$ u \in C\Bigl([0,T'];U_0 \Bigr) \cap C^1\left([0,T'];H^{s-1}(\R^d;\R^{1+d})\right)$$
and the solution $u$ depends continuously on the data.
\end{theorem}
We note that related results also hold on unbounded domains. However, one loses control of normal derivatives and must impose suitable boundary conditions. We do not pursue these technical issues here and refer the reader to \cite{OhnoShizutaYanagisawa1995} and the references therein for details.


\subsection{Organisation of the paper}

The remainder of the paper is organized as follows. In Section~\ref{se:prelim}, we introduce notation and recall preliminaries. Section~\ref{se:well-posedness} contains the proofs of Theorems~\ref{th:nonlin-well}, \ref{th:nonlin-well-refined}, and \ref{th:unbounded_domain}.
In Section~\ref{se:well-posedness_Westervelt}, we study the Westervelt-type system~\eqref{eq:main_system}, first establishing well-posedness for a linearized problem via energy estimates and then applying Banach’s fixed point theorem. Section~\ref{se:well-posedness_Kuznetsov} treats the Kuznetsov-type system in the inviscid case using \cite[Theorem~II]{Kato1975}.
The appendices collect auxiliary material: Appendix~\ref{se:derivation} contains the model derivation, Appendix~\ref{se:appendix_galerkin} a boundary relation used in the proof of Theorem~\ref{th:nonlin-well}, Appendix~\ref{se:kernel} the proof of the nonlinear coercivity inequality~\eqref{eq:kerlem_assump} and details on the optimality of Lemma~\ref{lemma:kernel}.

%
%
%
%
%
\section{Notation and Preliminaries} \label{se:prelim}
Throughout the current work, we simplify notation by omitting the explicit
$1+d$ superscript on Sobolev spaces. For example, $L^2(\Omega)$ may refer to $L^2(\Omega)^{1+d}=L^2(\Omega;\R^{1+d})$ or  $L^2(\Omega)^{(1+d) \times (1+d)} $. The context indicates whether the corresponding object is scalar, vector, or matrix valued. For any $L^2$ scalar product we write $(\cdot, \cdot)_2$. Again, it is clear from the context, which scalar product is meant. 

The partial derivative with respect to time is denoted by $\dt=\cdot_t$, while the partial derivative with respect to space in the direction $j=1, \dots, 3$ is written as $\dx{j}=\cdot_{x_j}$. The convolution operator $\conv$ is defined by
\[
(f \conv g)(t) = \int_0^t f(s)\, g(t-s)\, \mathrm{d}s
\]
for $t\geq 0$. If $f,g$ are vectors or matrices, we apply the convolution component wise. The vector of partial derivatives in space is denoted by $\nabla$, while $\cdot$ represents the scalar product between vectors, and the superscript ${\cdot}^\transpose$ denotes the transposition of a vector or a matrix.
We use the row-wise divergence of a matrix field 
$V \in \R^{m \times m}$, defined as
\[
(\Divv V)_i := \sum_{j=1}^m \partial_{x_j} V_{ij},
\qquad i = 1, \dots, m.
\]
The Frobenius inner product of matrices 
$A, B$ is denoted by $ A : B $. The identity matrix is denoted by $\id$, and the Jacobian matrix by $\jac$. The trace of a matrix $V$ is written as $\tr(V)$ and equals the sum of its diagonal entries. The (vector) Laplace operator is denoted as $\Delta$, $ a \times b$ stands for the cross product between vectors $a,b \in \R^3$.

Let $n$ denote the unit outward normal vector of $\partial \Omega$ and let $\dS$ be the surface measure. We use the shorthand notation $v_n := (\jac v)n$ as a normal derivative of a vector field $v$. We rely on the following integration by parts formulas; see \cite[{Sec.~1.5}]{grisvard2011elliptic}. 
\begin{lemma}\label{le:lapl_ibp}
Let $\Omega \subset \mathbb{R}^d$ be open, bounded with Lipschitz boundary, and let 
$u,v \in H^2(\Omega)$. Then
\begin{equation}
\label{eq:vector-green}
(u, \Delta v)_{2} = (\Delta u, v)_{2} + \int_{\partial \Omega}     u \cdot v_n  - v \cdot u_n \dS.
\end{equation}
\end{lemma}
\begin{proof}
The result is a consequence of \cite[Lem. 1.5.3.7]{grisvard2011elliptic}.
\end{proof}
\begin{lemma} \label{le:mat_ibp}
Let $\Omega \subset \mathbb{R}^d$ be open, bounded with Lipschitz boundary, let $V \in H^1(\Omega)$ be a matrix field, and $u \in H^1(\Omega)$.
Then
\begin{equation}\label{eq:ibp-matrix}
(u,\Divv V)_{2}
= -(\jac u , V)_{2} + \int_{\partial\Omega} u \cdot V n \dS .
\end{equation}
If, additionally, $V \in H^1(\Omega)\cap L^\infty(\Omega)$ is symmetric, $u \in H^1(\Omega) \cap L^\infty(\Omega)$, we obtain for $j =1, \dots, d$
\begin{equation}
2 ( V \dx{j} u, u)_2 = - ( \dx{j} V u, u)_2 +\int_{\partial \Omega} V u \cdot u n_j \dS.
\end{equation} 
\end{lemma}
\begin{proof}
The first equation is a consequence of Green's formula; see \cite[Thm. 1.5.3.1]{grisvard2011elliptic}. For the second equation we note that by the product rule and symmetry of $V$, we have
\[
\dx{j} \left(  u^\transpose V u \right) = \dx{j}u^\transpose V u + u^\transpose \dx{j} V u + u^\transpose V \dx{j} u = 2  u^\transpose V \dx{j} u +  u^\transpose \dx{j} V u .
\]
Note that, since $u, V \in L^\infty(\Omega)$, we have $u^\transpose V u \in H^1(\Omega)$, so that
by Green's formula, we obtain
\[
\int_\Omega 2  u^\transpose V \dx{j} u +  u^\transpose \dx{j}  V u \dxx = \int_{\partial \Omega} u^\transpose V u n_j \dS.
\] 
\end{proof}

These spaces allow us to define the solution space of the studied equation. We next introduce an elliptic regularity result that holds in the spaces $\Honefour, \Honefive$.
\begin{lemma} \label{le:elliptisc}
Let $\Omega \subseteq \R^d$ be bounded and have a boundary of class $C^{4,1}$. Then
\[ \label{eq:elliptic_reg}
\begin{aligned}
\| \psi \|_{H^4(\Omega)} \leq C_E \| \Delta^2 \psi \|_{L^2(\Omega)} \quad \forall \psi \in \Honefour, \\
\| \psi \|_{H^5(\Omega)} \leq C_E \| \jac \Delta^2 \psi \|_{L^2(\Omega)} \quad \forall \psi \in \Honefive
\end{aligned}
\]
with $C_E>0$.
\end{lemma}
\begin{proof}
The result follows from \cite[Thm.~2.5.1.1]{grisvard2011elliptic} arguing along the lines of \cite[Lem.~2.3]{meliani2025global}. Let $\psi \in \Honefour$. 
Writing $\phi=-\Delta \psi$ we have
\begin{equation}
\left\{ \begin{array}{ll}
- \Delta \phi = \Delta^2 \psi &  \mathrm{on } \qquad \Omega, \\
\phi = 0 &  \mathrm{on } \qquad \partial \Omega.
\end{array}
\right.
\end{equation}
The aforementioned \cite[Thm.~2.5.1.1]{grisvard2011elliptic} ensures 
$\| \phi\|_{H^2(\Omega)} \leq C \|\Delta ^2 \psi\|_{L^2(\Omega)}$, for some constant $C>0$. We also have
\begin{equation}
\left\{ \begin{array}{ll}
- \Delta \psi = \phi &  \mathrm{on } \qquad \Omega, \\
\psi = 0 &  \mathrm{on } \qquad \partial \Omega.
\end{array}
\right.
\end{equation}
Applying \cite[Thm.~2.5.1.1]{grisvard2011elliptic} once again, 
we obtain $\| \psi \|_{H^4(\Omega)} \leq \tild{C} \| \phi \|_{H^2(\Omega)}$ for some constant $\tild{C}$. The proof for $\Honefive$ follows similarly, exploiting the $C^{4,1}$ regularity of the boundary.
\end{proof}
We denote continuous embeddings between Banach spaces as $\embed$ and make use of the Sobolev embeddings 
\[ \label{eq:sob_emb}
H^1(\Omega) \embed L^{4}(\Omega), \quad H^2(\Omega) \embed L^{\infty}(\Omega)
\]
valid {for $d \leq 3$} and $\Omega \subset \R^d$ open, bounded with Lipschitz boundary, see \cite[Thm.~1.4.4.1]{grisvard2011elliptic}. 
We also use the fact that for $s \geq 2$, $H^s(\Omega)$ is a Sobolev Algebra for $d \leq 3$, that is, if $u,v \in H^s(\Omega)$, then the product $uv \in H^s(\Omega)$, see {\cite[Theorem 4.39]{adams2003sobolev}}. 

For convenience of notation, we introduce the commutator for a differential operator $S$ and a function $f$ as:
\[ \label{eq:commutator}
[S, f]:= S f - f S.
\]
In particular, for the choice $S= \Delta$, we obtain $$[\Delta, f](g) =\Delta(fg) - f \Delta g.$$

Our analysis is based on fixed-point iterations. The next lemma provides a general form of Banach's fixed point theorem tailored to situations where contractivity holds only in a weaker norm than the one controlling the a priori bounds. Such a formulation is standard in hyperbolic problems; see, for instance, \cite{kaltenbacher2024limiting} and \cite[Lemma~2.2]{majda1984compressible}.

We denote by $X^*$ the (topological) dual of a Banach space $X$.
\begin{lemma}\label{lemma:banach}
Let $X,Y$ be Banach spaces such that \( X \embed Y\) and assume there exist Banach spaces $Z,W$ such that $X=Z^*$, $Y=W^*$, with $Z$ separable.

Let $B\subset X$ be non-empty, bounded in $X$, and weak-$\star$ closed in $X$.
If $\mathfrak T:B\to B$ is a contraction with respect to the $Y$-norm, i.e., there exists $\gamma\in(0,1)$ such that
\begin{equation}\label{eq:lemma_cont}
\|\mathfrak T(x)-\mathfrak T(y)\|_Y \le \gamma \|x-y\|_Y \qquad \forall\,x,y\in B,
\end{equation}
then $\mathfrak T$ has a unique fixed point in $B$.
\end{lemma}

\begin{proof}
We intend to apply Banach's fixed point theorem to the metric space $M=(B, \| \cdot\|_Y)$. Since $\mathfrak T$ maps $B$ into itself and \eqref{eq:lemma_cont} gives contractivity in $M$, it only remains to prove that $M$ is complete.

Let $x_n \in B$ be Cauchy with respect to $\|\cdot\|_Y$. As $Y$ is complete, there exists $y\in Y$ such that $x_n\to y$ in $Y$. Since $B$ is bounded in $X=Z^*$ and $Z$ is separable, the Banach--Alaoglu theorem (see \cite[Thm. 3.2.1]{Buhler}) yields a subsequence $x_{n_k}$ and an $x\in X$ such that
$
x_{n_k}\rightharpoonup^\star x \text{ in } X.
$
Because $B$ is weak-$\star$ closed in $X$, we have $x\in B$.

On the other hand, norm convergence in $Y=W^*$ yields the weak-$\star$ convergence
$x_{n_k}\rightharpoonup^\star y$ in $Y$. Thus, since $ X \embed Y$, uniqueness of weak-$\star$ limits gives $x=y \in B$.
\end{proof}

\section{Well-posedness analysis} \label{se:well-posedness}
In the special case $\mathcal{J}=\delta_0$, the well-posedness analysis of \eqref{eq:main_system} is enabled through higher-order parabolic estimates. The higher regularity is needed to control the nonlinear terms via the Sobolev embeddings \eqref{eq:sob_emb} valid for $d \leq 3$.
Working with the solution norm
\begin{equation}
   E(u) = \| u \|_{L^\infty(0,T;H^1(\Omega))}^2 + \| u \|_{L^2(0,T;H^2(\Omega))}^2,
\end{equation}
the `linearized' nonlinear term $ \| A_j(z,x) \lapl u \|_{L^2(0,T;L^2(\Omega))} \leq E(z) \| \lapl u \|_{L^2(0,T;L^2(\Omega))}^2 $, is absorbed for $E(z)$ small enough; see \cite{QuadraticWave} for details.

For general kernels $\mathcal{J}$, the situation changes, since non-local behavior of the damping term is introduced. Performing the same steps as in the parabolic case and assuming coercivity of $\mathcal{J}$ in the form of \eqref{eq:kernel_assump}, one obtains the solution norm
\begin{equation}
   F(u) = \| u \|_{L^\infty(0,T;H^1(\Omega))}^2 + \| \mathcal{J} * u \|_{L^2(0,T;H^2(\Omega))}^2,
\end{equation}
which fails to provide control over $\| u \|_{L^2(0,T;H^2(\Omega))}^2$ needed to treat the nonlinear terms.

Kernels with sufficiently strong singularities can recover time regularity. For instance, in the case of Abel kernels \eqref{eq:frac_kernel}, one gains control of $\| u \|_{H^{-\alpha/2}(0,T;H^2(\Omega))}^2$ if $F$ is controlled, where \(H^{-\alpha/2}(0,T)\) is the dual space of the fractional space \(H^{\alpha/2}(0,T)\); see, e.g.,~\cite[Lemma 2.3]{eggermont1988galerkin}. However, we do not pursue this approach, as our goal is to establish well-posedness for a broader class of kernels. This leads to additional difficulties, since for very smooth kernels, no time regularity in  \(H^2(\Omega)\) is expected. Consequently, we must work with higher-order energy estimates that do not rely on the regularity induced by the Laplacian, thus placing \eqref{eq:main_system} in a hyperbolic framework.

For the Westervelt-type system, assuming $\mu=\kappa=0$, symmetrization of \eqref{eq:main_system} is achieved by multiplying \eqref{eq:main_system} with the diagonal matrix
\begin{equation} \label{eq:S}
S_W(u) = S_W(p,v) = 
\begin{pmatrix}
1 & 0 \\
0 & \frac{1 + \alpha + \lambda p}{1 + \beta } \id_{\R^d}
\end{pmatrix} \in \R^{(1+d) \times (1+d)},
\end{equation}
while for the Kuznetsov-type system we refer to Section \ref{se:well-posedness_Kuznetsov} for details. The symmetrization process introduces nonlinearity in front of the Laplacian operator when $\mathcal{J} \neq 0$. This difficulty is overcome by the nonlinear coercivity result (Lemma~\ref{lemma:kernel}) under suitable assumptions on $S$. 

Note that while $\vartheta=\kappa$ is physically motivated for \eqref{eq:first_order}, the parabolic case allows for a wider range of parameters; see~\cite{QuadraticWave}. 

Before turning to the proofs, we briefly note that high regularity of solutions is standard in the theory of quasilinear symmetric hyperbolic systems; see, e.g.,~\cite{Kato1975}. In the present work, we assume $H^4$ regularity of the initial data. This choice ensures sufficient regularity of the solution to control the nonlinearities and facilitates the cancellation of boundary terms under Dirichlet conditions; see, e.g.,~\eqref{eq:thirdterm}.

\subsection{Well-posedness of Westervelt-type equation} \label{se:well-posedness_Westervelt}
In this section, we establish the well-posedness of the nonlinear problem \eqref{eq:main_system} when $\mu=\kappa=0$. To this end, our strategy is as follows. 
We linearize \eqref{eq:main_system} at a sufficiently regular $v:(0,T) \times \Omega \to \R^{1+d}$ yiedling a linear system with variable coefficients 
\[ \label{eq:resulting_system}
A_0 u_t + \sum_{j=1}^d A_j(v,x) u_{x_j} - \mathcal{J} \conv \lapl u = f.
\]
The resulting system \eqref{eq:resulting_system} is in general not symmetric.
The symmetrization introduces variable coefficients (whose matrix forms are denoted below by $M_j$ with $0\leq j\leq d $). 
\subsubsection{Linearized well-posedness}
Motivated by multiplication of \eqref{eq:resulting_system} by the symmetric matrix \(S_W\) from \eqref{eq:S}, we consider the following linearized system with variable coefficients
\begin{equation}\label{eq:main_system_linear}
\left\{
\begin{array}{ll}
{M}_0 u_t + \sum_{j=1}^d \lin_j  u_{x_j}  - {J_0} (\mathcal{J} \ast \Delta u ) = f & \text{on } (0,T) \times \Omega,\\
u(0) = u_0 & \text{on } \Omega,\\
u = 0 & \text{on } (0,T) \times \partial \Omega,
\end{array}
\right.
\end{equation}
where, for each $1 \le j \le d$, $M_j \in \R^{(1+d)\times(1+d)}$ is a symmetric matrix-valued function depending on space and time.
$M_0, J_0$ are assumed to be symmetric positive definite, depending on space and time. $J_0$ is additionally assumed to be diagonal. 
The estimates obtained for the linearized system~\eqref{eq:main_system_linear} allow one to construct suitable balls in these function spaces, on which a fixed point argument can be applied. This ensures the unique solvability of the nonlinear system~\eqref{eq:main_system}, see Section~\ref{sec:fp_arg}. We have the following linearized well-posedness theorem.
\begin{theorem}\label{thm:lin-well}
Let $T>0$ and let {$\Omega\subset \R^d$}, $d \leq 3$, be a bounded domain with boundary of class $C^{4,1}$. Assume $\mathcal{J} \in L^1(0,T) $ satisfies Assumption \eqref{as:kernel}
and that the data satisfies
\[
u_0 \in {\Honefour}, 
\quad 
{f \in \mathcal{V}(0,T) := L^1 \! \left(0,T;\Honefour\right)
\cap L^{\infty}\! \left(0,T;\Honetwo \right)} .
\]
For each $j=0, 1, \dots, d$ assume that the matrices $M_j \in \calU(0,T) $ are symmetric, $M_0$ is positive definite, and that $J_0 \in \calU(0,T)$ is diagonal, positive definite, and satisfies 
\[ \label{eq:kernel_coff}
\| \dt J_0 \|_{L^{\infty}(0,T;L^\infty(\Omega))}< C_K \inf_{t \in (0,T), \, x \in \Omega} \lambda_{min} ( J_0(t,x) )
\]
for some $C_K>0$ depending on $\mathcal{J},T$.
Then, equation~\eqref{eq:main_system_linear} admits a unique solution $u \in \mathcal{U}_{\diamondsuit}(0,T)$ with \emph{a priori} estimate 
\[ \label{eq:apriori_linear}
\| u \|_{\calU(0,T)} \leq C \left( \| u_0 \|_{\Honefour} + \|  f \|_{\mathcal{V}(0,T)} \right)
\]
for a constant $C>0$ that depends only on $\lin_j, J_0, \mathcal{J}, T, \Omega, d$.
\end{theorem}
\begin{proof}
The proof of Theorem~\ref{thm:lin-well} is based on the Galerkin approximation combined with energy estimates and is carried out in the four following steps.
\begin{enumerate}[(i)]
\item Construct approximate finite-dimensional solutions using the Galerkin method.
\item Derive uniform \emph{a priori} estimates for these approximate solutions.
\item Pass to the limit, exploiting the compactness properties of the solution space.
\item Prove that the resulting limit is the unique solution of~\eqref{eq:main_system_linear} in the solution space.
\end{enumerate}
\subsubsection*{Step (i): Construction of approximate solutions}
Let $W := \{w_k\}_{k=1}^\infty \subset H_0^1(\Omega)$ denote the complete set of $L^2$-normalized eigenfunctions of the Dirichlet Laplacian
$ - \Delta : \Honetwo \to L^2(\Omega).$
By construction, $W$ forms an orthonormal basis of $L^2(\Omega)$ and an orthogonal basis of $H_0^1(\Omega)$.  

For given $N \in \mathbb{N}$, we aim to approximate solutions of \eqref{eq:main_system_linear} by functions of the form
\[ \label{eq:galerkin_approx}
u^N(t,x) = \sum_{k=1}^N \sum_{\ell=1}^{1+d} \xi_k^{N, \ell}(t) e_\ell w_k(x), \quad \xi_k^{N,\ell}(t) \in \mathbb{R}, \quad t \in (0,T), \, x \in \Omega,
\]
where $e_\ell \in \R^{1+d}$ denote unit normal vectors.
The Galerkin projection of~\eqref{eq:main_system_linear} onto the finite dimensional subspace 
\[
U_N := \operatorname{span}\{w_k e_\ell : k = 1, \dots, N; \ell = 1, \dots, 1+d\}  {\, \subset C^\infty(\Omega)}
\cap  H^1_0(\Omega)
\] 
is given by
\[ \label{eq:galerkin_projection}
\left\{
\begin{array}{ll}
(\lin_0 u^N_t, w_i e_\ell)_2 + \sum_{j=1}^d (\lin_j u^N_{x_j}, w_i e_\ell)_2 - (\dam \mathcal{J} \ast \lapl u^N,   w_i e_\ell )_2 = (f, w_i e_\ell)_2,
\\
u^N(0) = u_0^N
\end{array}
\right.
\]
for all $i=1, \dots, N$ and $\ell=1, \dots, 1+d$ with initial condition $u_0^N$, obtained by $L^2$-projection onto $U_N$. 
Denoting $\xi_k^N:=(\xi_k^{N,1}, \dots, \xi_k^{N,1+d}  )^\transpose \in \R^{1+d}$ and, further, 
\[
\boldsymbol{\xi} := \left( \xi_1^N, \dots, \xi_N^N \right)^\transpose \in \R^{(1+d)N}
\] 
we obtain the matrix form of~\eqref{eq:galerkin_projection} as
\[ \label{eq:galerkin_matrixform}
\mathbf{M}(t) \dot{ \boldsymbol{\xi}}(t)  +  \mathbf{A}(t)  \boldsymbol{\xi}(t) + \mathbf{J}(t) \conv \boldsymbol{\xi}(t)  = \mathbf{f}(t)
\]
with
{initial data $\boldsymbol{\xi}(0) = \boldsymbol{\xi}_0$ given as the vector of coordinates of the approximate initial data $u_0^N$.} {In~\eqref{eq:galerkin_matrixform}}, the matrices and vectors involved are {given} by
\begin{equation}
\begin{aligned}
&\mathbf{M}(t) = \id_{\R^N} \otimes \lin_0(t), \\  &\mathbf{A}_{(i-1)(1+d)+\ell, (k-1)(1+d)+\ell'}(t) = \sum_{j=1}^d ( \lin_j(t) e_{\ell'} \partial_{x_j} w_k, w_i e_\ell)_2, \\ 
&\mathbf{J}(t) = S(t) \otimes ({ \id_{\R^{1+d}} \mathcal{J}(t) }) \quad \mathrm{with }\quad 	S_{ik}(t)= (\grad w_i , \grad ( \dam(t) w_k ))_2 \in \R^{N \times N}, \\ 
&\mathbf{f}_{(i-1)(1+d)+\ell}(t) = (f(t) e_l, w_i)_2 .
\end{aligned}
\end{equation}
Here, $\otimes$ denotes the tensor product and $I_{\R^n}$ denotes the identity matrix $I_{\R^n} : \R^n \to \R^n$ and $\mathbf{1}:= (1, \dots, 1)^\transpose \in \R^{(1+d)N}$. 

To show well-posedness of~\eqref{eq:galerkin_matrixform}, we set $\boldsymbol{\chi} = \dot{\boldsymbol{\xi}}$, which implies $\boldsymbol{\xi} = \mathbf{1} \conv \boldsymbol{\chi} + \boldsymbol{\xi}_0$. We use the fact that $\mathbf{M}$ is positive definite, to rewrite~\eqref{eq:galerkin_matrixform} as
\[ \label{eq:galerkin_volterra}
\boldsymbol{\chi}(t) +  \int_0^t \mathbf{L}(t,s) \boldsymbol{\chi}(s) \ds  = \mathbf{g}(t),
\]
with 
\[
\mathbf{L}(t,s) := \mathbf{M}^{-1}  \left[ (\mathbf{J} \conv { (\mathbf{1} \otimes \mathbf{1} }) (t-s) + \mathbf{A}(t) \right]\] 
and 
\[\mathbf{g}(t):=\mathbf{M}^{-1} \left[\mathbf{f}(t) - (\mathbf{J}\conv \boldsymbol{\xi}_0)(t) -  \mathbf{A}(t) \boldsymbol{\xi}_0\right].
\]
Equation~\eqref{eq:galerkin_volterra} can be interpreted as a non-convolution Volterra equation, see \cite[Ch. 9]{Gripenberg1990}.  
We first note that the non-convolution kernel $\mathbf{L}(t,s)$ is of \emph{$L^p$-type}
\footnote{Here, a kernel of $L^p$-type should not be confused with a $L^p$ function.}, 
i.e., the operator
\[
y \;\mapsto\; \int_0^t |\mathbf{L}(t,s)| \, y(s) \ds
\]
is continuous from $L^p(0,T)$ into itself for $p \in [1,\infty]$, see \cite[Ch.~9, Def.~2.2]{Gripenberg1990} and \cite[Ch.~9, Eq.~2.2]{Gripenberg1990}.  
This property comes from the fact that $\mathcal{J} \in L^1(0,T)$ and $\mathbf{A} \in L^\infty(0,T)$. Consequently, the kernel $\mathbf{L}$ admits a resolvent of $L^p$-type, by \cite[Ch.~9, Cor.~3.14]{Gripenberg1990}. 

The existence and uniqueness of a solution $\boldsymbol{\chi} \in L^2(0,T)$ to \eqref{eq:galerkin_volterra} follow from \cite[Ch.~9, Thm.~3.6]{Gripenberg1990}, together with the fact that $\mathbf{g} \in L^2(0,T)$ . 
Thus, it follows that $\boldsymbol{\xi} \in H^1(0,T)$, which in turn ensures that 
$ u^N \in H^{{1}}\big(0,T; C^\infty(\Omega) \cap H_0^1(\Omega)\big).$

\subsubsection*{Properties of approximate solutions}
For future reference, we note that since $\Omega$ has a boundary of class $C^{4,1}$, elliptic regularity implies $u^N \in H^1\left(0,T;\Honefive\right)$, see, for instance, in \cite[Thm.~2.5.1.1]{grisvard2011elliptic}.
Furthermore, at the Galerkin approximation level, by the construction of the function space $W$, we have $\lapl^j u^N \in   H^1\left(0,T;\Honefive \right)$ and
\[\label{eq:zero_bnd_properties}
u^N|_{\partial \Omega} = (-\Delta)^j u^N|_{\partial \Omega} = 0
\] 
for any natural number $j \geq 0$ for all $u^N \in U_N$. This justifies using partial integration as given by Lemma \ref{le:lapl_ibp}, \ref{le:mat_ibp} in the next sections.  
The same holds for $u_t^N$ and $\mathcal{J} \conv u^N$, i.e.,
\[ \label{eq:zero_bnd_properties_extended} 
(-\Delta)^j u_t^N|_{\partial \Omega}   = \mathcal{J} \conv (-\Delta)^j u^N|_{\partial \Omega} = 0  . 
\]
Furthermore, from~\eqref{eq:zero_bnd_properties}, ~\eqref{eq:zero_bnd_properties_extended}, and $f \in L^1(0,T;\Honefour)$ the identities 
\[\label{ineq:bnd_value_convection_term} 
\sum_{j=1}^d  \int_{\partial \Omega} \lin_j u_{x_j}  \cdot \jac \phi^N n  \dS =0, 
\]
\[ \label{ineq:bnd_value_convection_term_lapl}
\int_{\partial \Omega} \lapl \left( \lin_0 u_t^N + \sum_{j=1}^d \lin_j u_{x_j}^N - \dam (\mathcal{J} \conv \lapl u^N) \right) \cdot \jac \phi^N n \dS = 0
\]
holding for all $\phi^N \in U_N$ can be derived from the projected equation \eqref{eq:galerkin_projection}. The details are explained in the {appendix} in section \ref{se:appendix_galerkin}.
\subsubsection*{Step (ii): Establishing uniform estimates}	
Fix $N \in \N$. We test equation~\eqref{eq:galerkin_projection} with $\lapl^4 u^N \in U_N$ to obtain
\begin{equation} \label{eq:third_order_energy}
(\lin_0 u_t^N, \lapl^4 u^N)_2 
+ \sum_{j=1}^d (\lin_j u_{x_j}^N ,  \lapl^4 u^N)_2 
- ( \dam (\mathcal{J} \conv \lapl u^N),  \lapl^4 u^N)_2 = (f, \lapl^4 u^N)_2 .
\end{equation}
We now estimate each term in~\eqref{eq:third_order_energy} from above or below. For better readability, we leave out the index $N$ in $u^N$ when deriving the uniform estimates. 
The first term in~\eqref{eq:third_order_energy}, after integrating by parts twice with Lemma~\ref{le:lapl_ibp} and using~\eqref{eq:zero_bnd_properties}, \eqref{eq:zero_bnd_properties_extended}, as well as the short hand notation $\laplu=\Delta^2 u$ gives
\[ \label{eq:firstterm}
\begin{aligned}
(\lin_0 u_t,  \lapl^4 u )_2 
&= ( \lapl (\lin_0  u_t),  \lapl^3 u)_2 + \int_{\partial \Omega}   \lin_0 u_t \cdot (\lapl^3 u)_n - (\lin_0 u_t)_n \cdot \lapl^3 u  \dS
\\
&= ( \lapl^2 (\lin_0   u_t), v)_2 + \int_{\partial \Omega} \lapl (\lin_0  u_t) \cdot   \laplu_n - ( \lapl (\lin_0  u_t))_n \cdot \laplu \dS 
\\
& = (  \lin_0 \lapl^2  u_t, v)_2 + ([\Delta^2, M_0](u_t), \laplu)_2 
+ \int_{\partial \Omega} \lapl (\lin_0  u_t) \cdot   \laplu_n \dS
\\
& = \frac{1}{2} \frac{d}{dt} ( \lin_0^\half \laplu, \lin_0^\half \laplu)_2 - \half (\dt M_0 \laplu, \laplu)_2  + ([\Delta^2, M_0](u_t), \laplu)_2 \\
&+ \int_{\partial \Omega} \lapl (\lin_0  u_t) \cdot   \laplu_n \dS,
\end{aligned} 
\]
where the remaining boundary term is handled later by \eqref{ineq:bnd_value_convection_term_lapl} and $[\cdot, \cdot]$ denotes the commutator introduced in \eqref{eq:commutator}.
For the second term in \eqref{eq:third_order_energy}, integrating by parts with Lemma~\ref{le:lapl_ibp}, \ref{le:mat_ibp}, and using the boundary information~\eqref{eq:zero_bnd_properties}, \eqref{eq:zero_bnd_properties_extended},~\eqref{ineq:bnd_value_convection_term} gives
\[ \label{eq:thirdterm}
\begin{aligned}
& \sum_{j=1}^d  ( \lin_j u_{x_j}, \lapl^4 u )_2 \\
=  &\sum_{j=1}^d ( \lapl ( \lin_j u_{x_j} ) ,  \lapl^3 u)_2 
+ \! \int_{\boundary \Omega} \! {\lin_j u_{x_j} \cdot (\lapl^3 u)_n } - ( \lin_j u_{x_j} )_n \cdot \lapl^3 u \dS
\\
= 
&\sum_{j=1}^d( \lapl^2 (\lin_j u_{x_j} ),  \laplu )_2 + 
\int_{\boundary \Omega} \lapl ( \lin_j u_{x_j} ) \laplu_n - (\lapl ( \lin_j u_{x_j} )_n \laplu \dS \\
= &\sum_{j=1}^d ( \lin_j  \laplu_{x_j} ,  \laplu )_2 +  ( [\lapl^2, \lin_j]( u_{x_j}),  \laplu )_2  + 
\int_{\boundary \Omega} \lapl ( \lin_j u_{x_j} )  \laplu_n\dS
\\
=  &\sum_{j=1}^d (  - \half \dx{j} \lin_j  \laplu +   [\lapl^2,\lin_j]( u_{x_j} ),  \laplu )_2  + 
\int_{\boundary \Omega} \lapl ( \lin_j u_{x_j} )  \laplu_n + \underbrace{
\half \lin_j v \cdot v n_j
}_{=0}\dS,
\end{aligned}
\]
{where, in the last line, we argue the vanishing of the second boundary term thanks to~\eqref{eq:zero_bnd_properties_extended}. In order to treat the remaining boundary part, we intend to use~\eqref{ineq:bnd_value_convection_term_lapl}. To this end, we need to exploit the remaining $M_0$ and $J_0$ terms in~\eqref{eq:third_order_energy}. }

Using the partial integration Lemmata~\ref{le:lapl_ibp}--\ref{le:mat_ibp}, and the boundary information~\eqref{eq:zero_bnd_properties}--\eqref{eq:zero_bnd_properties_extended}, the third term in \eqref{eq:third_order_energy} becomes
\[ \label{eq:secondterm}
\begin{aligned}
&\quad - (  \dam \mathcal{J} \conv \lapl u,  \lapl^4 u )_2 
\\
&= - ( \lapl (\dam \mathcal{J} \conv \lapl u ), \lapl^3 u)_2 + \int_{\partial \Omega} \dam \mathcal{J} \conv \lapl u \cdot (\lapl^3 u)_n -  (\dam \mathcal{J} \conv \lapl u)_n \cdot \lapl^3 u \dS 
\\
&= (  \jac \Delta (\dam \mathcal{J} \conv \lapl u) , \jac \lapl^2 u)_2 - \int_{\partial \Omega} \Delta ( \dam \mathcal{J} \conv \lapl u ) \cdot (\lapl^2 u)_n \dS 
\\
&= ( \dam \jac   \mathcal{J} \conv \laplu , \jac \laplu)_2 + ( [\jac \Delta, \dam](\mathcal{J} \conv \Delta u)), \jac \laplu)_2 - \int_{\partial \Omega} \Delta ( \dam \mathcal{J} \conv \lapl u ) \cdot \laplu_n \dS 
\\
&= ( \dam \jac  \mathcal{J} \conv \laplu , \jac \laplu)_2 - ( \jac \cdot [\jac \Delta, \dam](\mathcal{J} \conv \Delta u), \laplu)_2 \\
&\qquad \qquad- \int_{\partial \Omega} \Delta ( \dam \mathcal{J} \conv \lapl u ) \cdot \laplu_n - \underbrace{[\jac \Delta, \dam](\mathcal{J} \conv \Delta u) n \cdot v }_{=0} \dS
.
\end{aligned}
\]
For the right hand side in \eqref{eq:third_order_energy}, we have since $f \in L^1(0,T;\Honefour)$
\[ \label{eq:forthterm}
\begin{aligned}
(f , \lapl^4 u)_2 &= (\lapl f, \lapl^3 u)_2 + \int_{\boundary \Omega} f \cdot (\lapl^3 u)_n - (f)_n \cdot \lapl^3 u \, dS\\
&= (\lapl^2 f, \lapl^2 u)_2 + \int_{\boundary \Omega} \lapl f \cdot (\lapl^2 u)_n - (\lapl f)_n \cdot \lapl^2 u \, dS \\
& = (\lapl^2 f, \laplu)_2.
\end{aligned}
\]
Combining \eqref{eq:firstterm}, \eqref{eq:secondterm}, \eqref{eq:thirdterm}, and \eqref{eq:forthterm} together with boundary information~\eqref{ineq:bnd_value_convection_term_lapl}, we obtain 
\begin{equation} \label{eq:third_order_energy_summary}
\begin{aligned}
&\half \frac{d}{dt} \|  \lin_0^\half \laplu \|_ {L^2(\Omega)}^2 + (\dam  \jac  \mathcal{J} \conv \laplu, \jac \laplu)_2  \\
&\leq  \big|( \sum_{j=0}^d [\Delta^2 , M_j](u_{x_j}) + \half \dx{j} M_j v + \jac \cdot [\jac \Delta, \dam](\mathcal{J} \conv \Delta u) + \lapl^2 f , \laplu)_2 \big| 
\end{aligned}
\end{equation}
with the convention $u_{x_0}=u_t$. Integrating \eqref{eq:third_order_energy_summary} from $0$ to $t\in (0,T)$ we obtain using Lemma \ref{lemma:kernel}, since $\dt J_0 \in L^\infty(0,T;H^2(\Omega)) \embed L^\infty(0,T;L^\infty(\Omega)) $ and due to \eqref{eq:kernel_coff}, 
\[ \label{eq:integrated_energy}
\begin{aligned}
&\half \|  \lin_0^\half \lapl^2 u(t) \|_{L^2(\Omega)}^2 + C_{\mathcal{J}} \|\jac  \mathcal{J} \conv \lapl^2 u \|^2_{L^2(0,t;L^2(\Omega))} \\
&\leq   \half \|  \lin_0^\half \lapl^2 u_0 \|_{L^2(\Omega)}^2  +  \int_0^t \| \mathfrak{P}(s) \|_{L^2(\Omega)} \| \lapl^2 u(s) \|_{L^2(\Omega)} \ds,
\end{aligned}
\]
where $\mathfrak{P}=\sum_{j=0}^d [\Delta^2,M_j](u_{x_j}) + \half \dx{j} M_j \Delta^2 u + \jac \cdot [\jac \Delta, \dam]( \mathcal{J} \conv \Delta u) + \lapl^2 f$ denotes the remaining terms. We estimate these using elliptic regularity of Lemma \ref{le:elliptisc}, Sobolev embeddings and algebras
\[
\half  \|   \dt M_0 \Delta^2u \|_{L^2(\Omega)} \leq C_R \|   \dt  M_0 \|_{H^3(\Omega)} \| \Delta^2 u \|_{L^2(\Omega)} ,
\]
\[
\| [\Delta^2, M_0] (u_{t}) \|_{L^2(\Omega)}  \leq C_R \|  M_0 \|_{H^4(\Omega)} \| \jac \Delta u_t \|_{L^2(\Omega)},
\]
\[
\sum_{j=1}^d \| [\Delta^2, M_j] (u_{x_j}) \|_{L^2(\Omega)} + \|   \dx{j} M_j \Delta^2u \|_{L^2(\Omega)} \leq \sum_{j=1}^d C_R \|  M_j \|_{H^4(\Omega)} \|  \Delta^2 u \|_{L^2(\Omega)},
\]
\[
\| \jac \cdot [\jac \Delta, \dam]( \mathcal{J} \conv \Delta u) \|_{L^2(\Omega)} \leq C_R \|J_0 \|_{H^4(\Omega)} \| \jac \mathcal{J} \conv \Delta^2 u \|_{L^2(\Omega)}
\]
with some constant $C_R>0$ depending on elliptic regularity and embedding constants.
Thus, we have using Young's inequality twice
\[ \label{eq:estimate_P}
\begin{aligned}
\int_0^t \| \mathfrak{P} \|_{L^2(\Omega)} \| \lapl^2 u \|_{L^2(\Omega)} \ds \leq \frac{1}{4 \varepsilon_1} \| f \|_{L^1(0,T;\Honefour)}^2 + \varepsilon_1 \| \Delta^2 u \|_{L^\infty(0,t;L^2(\Omega))}^2
\\ 
+  \int_0^t C_R \left(  \|   \dt  M_0 \|_{H^3(\Omega)}  + \sum_{j=1}^d \|  M_j \|_{H^4(\Omega)} + \frac{1 }{2 C_{\mathcal{J}}}\| J_0 \|_{H^4(\Omega)}  \right) \| \Delta^2 u \|^2_{L^2(\Omega)}
\\
+ \frac{C_{\mathcal{J}}}{2}\| \jac \mathcal{J} \conv \Delta^2 u\|^2_{L^2(\Omega)}  + C_R \| M_0 \|_{H^4(\Omega)} \| \jac \Delta u_t \|_{L^2(\Omega)}  \| \Delta^2 u \|_{L^2(\Omega)} \ds 
\end{aligned}
\]
with $\varepsilon_1>0.$
In order to proceed, we need to bound $u_t$ in $H^3(\Omega)$. By \eqref{eq:galerkin_projection} we have for any $\phi \in U_N$
\newcommand{\Honethreedual}{H^{-3}(\Omega)}
\[ \label{eq:energy_esimiate_ut}
\begin{aligned}
(\lin_0 u_t, \phi)_2 &\leq | ( \dam \mathcal{J} \conv \lapl u - \sum_{j=1}^d \lin_j u_{x_j} + f, \phi)_2 |  \\
&\leq  \| \dam \mathcal{J} \conv \lapl u  - \sum_{j=1}^d \lin_j u_{x_j} +  f \|_{\Honethree} \| \phi \|_{\Honethree^*} .
\end{aligned}
\]
Thus, using triangle inequality, Sobolev algebra properties of $H^3(\Omega)$, and elliptic regularity Lemma \ref{le:elliptisc} up to $H^5(\Omega)$
\[ \label{eq:estimate_u_t_H3}
\begin{aligned}
\| \lin_0 u_t \|_{\Honethree} &\leq C_R \Bigl(\|  \dam \|_{H^3(\Omega)} \| \jac \mathcal{J} * \lapl^2 u \|_{L^2(\Omega)} \\
&+ \sum_{j=1}^d \|  M_j \|_{H^3(\Omega)} \| \Delta^2 u \|_{L^2(\Omega)} \Bigr) + \| f \|_{\Honethree}.
\end{aligned}
\]
Since $M_0 \in L^\infty(0,T;H^3(\Omega))$ is positive definite, we obtain, due to Sobolev algebra properties, 
\[
\begin{aligned}\label{est::ut_H3}
\| u_t \|_{\Honethree} &= \| M_0^{-1} M_0 u_t \|_{\Honethree} \leq C_R \| M_0^{-1} \|_{H^3(\Omega)} {\color{black}\| \lin_0 u_t \|_{\Honethree}} \\ &\leq C_R C_{M_0} \| M_0 \|_{H^3(\Omega)} {\color{black}\| \lin_0 u_t \|_{\Honethree} },
\end{aligned}
\]
where $C_{M_0}$ depends on the smallest eigenvalue of $M_0$. 
Combing \eqref{eq:integrated_energy}, \eqref{eq:estimate_P} and \eqref{eq:estimate_u_t_H3} with Young's inequality we get
\[ \label{eq:energy_estimate1}
\begin{aligned}
\half c_{\lin_0} \|   \lapl^2 u(t) \|_{L^2(\Omega)}^2&
+ \half C_{\mathcal{J}} \|\jac  \mathcal{J} \conv \lapl^2 u \|^2_{L^2(0,t;L^2(\Omega))} 
\leq  \frac{1}{4 \varepsilon_1} \| f \|^2_{L^1(0,T;\Honefour)} + \varepsilon_1 \| \Delta^2 u \|_{L^\infty(0,T;L^2(\Omega))}^2 \\
&+ \half \| \lin_0 \|_{L^\infty(0,T;L^\infty(\Omega))} \|  \lapl^2 u_0 \|_{L^2(\Omega)}^2 
\\&+  \int_0^t C_R \left(  \| \dt  M_0 \|_{H^3(\Omega)}   + \sum_{j=1}^d \|  M_j \|_{H^4(\Omega)} + \frac{1 }{2 C_{\mathcal{J}}}\|  J_0 \|_{H^4(\Omega)}  \right) \|\Delta^2 u \|^2_{L^2(\Omega)}  
\\
& 
+ C_R^3 C_{M_0} \| M_0 \|_{H^4(\Omega)}^2  \left( \frac{1}{C_{\mathcal{J}}}\|  \dam \|_{H^4(\Omega)} + \frac{1}{4 C_R}
+ \sum_{j=1}^d \|  M_j \|_{H^4(\Omega)}  \right) \|\Delta^2 u \|_{L^2(\Omega)}^2  \ds
\\
&+ \frac{1}{4} C_{\mathcal{J}} \| \jac \mathcal{J} * \lapl^2 u \|_{L^2(0,t;L^2(\Omega))}^2 +  \| f \|_{L^2(0,T;\Honethree)}^2
\end{aligned}
\]
with $c_{M_0}$ the minimal eigenvalue of $M_0$. Then, applying Grönwall's inequality in \eqref{eq:energy_estimate1}, and taking the supremum over $(0,T)$ we get
\[ \label{eq:energy_estimate_gron}
\begin{aligned}
\half c_{\lin_0} \|   \lapl^2 u \|_{L^\infty(0,T;L^2(\Omega))}^2
&+ \quater C_{\mathcal{J}} \|\jac  \mathcal{J} \conv \lapl^2 u \|^2_{L^2(0,T;L^2(\Omega))} 
\\ &\leq \Bigl( \frac{1}{4 \varepsilon_1} \| f \|^2_{L^1(0,T;\Honefour)} + \half   \| \lin_0 \|_{L^\infty(0,T;L^\infty(\Omega))}\| \lapl^2 u_0 \|_{L^2(\Omega)}^2 
\\
& \hphantom{ppp}+ \varepsilon_1 \| \Delta^2 u \|_{L^\infty(0,T;L^2(\Omega))}^2 
+  \| f \|^2_{L^2(0,T;\Honethree)} \Bigr)  \exp \left( C_Q \int_0^T  Q_{}(s)  \ds \right)
\end{aligned}
\]
with \[ \label{eq:Q}
Q=   \| \dt  M_0 \|_{H^3(\Omega)} + \|  M_0 \|_{H^4(\Omega)}^2 + \left( 1 + \|  M_0 \|_{H^4(\Omega)}^2 \right) \left(    \sum_{j=1}^d \|  M_j \|_{H^4(\Omega)} + \|  J_0 \|_{H^4(\Omega)} \right),
\]
where $C_Q$ only depends on $C_\mathcal{J}, C_R, C_{M_0}$. 
Choosing $\varepsilon_1 = \quater c_{M_0} \exp\left(- C_Q \int_0^T Q(s) \ds\right) $ we
can absorb the term
and defining the energy
$$
\E[u](t) = \frac{c_{M_0}}{2} \sup_{s \in (0,t)}\|  u(s) \|^2_{L^2(\Omega)} + \frac{C_{\mathcal{J}}}{2}\int_0^t \| \jac (\mathcal{J} * u)(s) \|_{L^2(\Omega)}^2  \ds  , 
$$
we have
\[ \label{eq:energy_estimate_H4}
\begin{aligned}
 \E[\Delta^2 u](T) \leq C_G \Big(&  \| M_0 \|_{L^\infty(0,T;L^\infty(\Omega))}\| u_0 \|_{\Honefour}^2  + \| f \|^2_{L^1(0,T;\Honefour)} \\ &+ \| f \|_{L^2(0,T;\Honethree)}^2 \Big) 
\exp{\left( \int_0^T Q_{}(s) \ds  \right) }
\end{aligned}
\]
with $C_G>0$. 
We note that $u \in W^{1, \infty}(0,T;\Honetwo)$ follows, since from \eqref{eq:energy_esimiate_ut} by Young's inequality, Sobolev algebra $H^2(\Omega)$, and elliptic regularity
\[ \label{eq:estimate_u_t_H2}
\begin{aligned}
&\| M_0 u_t \|_{L^\infty(0,T;\Honetwo)} \leq \| J_0 \|_{L^\infty(0,T;H^2(\Omega))} \| \mathcal{J} \|_{L^1(0,T)} \| \Delta^2 u \|_{L^\infty(0,T;L^2(\Omega))} 
\\
&+ \sum_{j=1}^d \|  M_j \|_{L^\infty(0,T;H^2(\Omega))} \| \Delta^2 u \|_{L^\infty(0,T;L^2(\Omega))} + \| f \|_{L^\infty(0,T;\Honetwo)}.
\end{aligned}
\]
Similarly as before, we have
\[
\| u_t \|_{\Honetwo} \leq C_R \| M_0^{-1} \|_{H^2(\Omega)} \| \lin_0 u_t \|_{\Honetwo} \leq C_R C_{M_0} \| M_0 \|_{H^2(\Omega)} \| \lin_0 u_t \|_{\Honetwo}
\]
with {the constants $C_R$, $C_{M_0}$ related to the Algebra properties of $H^2(\Omega)$ and to the minimal eigenvalue of $M_0$ (similarly to~\eqref{est::ut_H3})}. 
This gives
\[ \label{eq:estimate_u_t_H2_v2}
\begin{aligned}
\| u_t \|_{L^\infty(0,T;\Honetwo)} \leq C_2 \|M_0\|_{L^\infty(0,T;H^2(\Omega) )} \Bigl( \| f \|_{L^\infty(0,T;\Honetwo)} +  \bigl( \| J_0 \|_{L^\infty(0,T;H^2(\Omega)} \\
+ \sum_{j=1}^d \|  M_j \|_{L^\infty(0,T;H^2(\Omega))} \bigr)  \sqrt{ \E[\Delta^2 u](T)}  \Bigr) .
\end{aligned}
\]
Finally, from \eqref{eq:estimate_u_t_H3} we infer
\[ \label{eq:estimate_u_t_H3_v2}
\begin{aligned}
    \| u_t \|_{L^2(0,T;\Honethree)} \leq C_3 \|M_0\|_{L^\infty(0,T;H^3(\Omega))} \Bigl( \| J_0 \|_{L^\infty(0,T;\Honethree)}  \sqrt{ \E[\Delta^2 u](T)} \\
    + \sum_{j=1}^d \| M_j \|_{L^2(0,T;H^3(\Omega))}  \sqrt{ \E[\Delta^2 u](T)}  + \| f \|_{L^2(0,T;\Honethree)} \Bigr) .
\end{aligned}
\]
Thus, combining \eqref{eq:estimate_u_t_H2_v2} and \eqref{eq:estimate_u_t_H3_v2} with the energy estimate \eqref{eq:energy_estimate_H4} we obtain the a priori bound \eqref{eq:apriori_linear} for $u^N$.

Note that we have used that $ \mathcal{V}(0,T) \embed L^2(0,T;\Honethree) $ by interpolation.
\subsubsection*{Step (iii): Passing to the limit}
The previous estimates \eqref{eq:estimate_u_t_H2_v2} and \eqref{eq:estimate_u_t_H3_v2} with \eqref{eq:energy_estimate_H4}, imply that the sequence of approximate solutions $u^N$, once again emphasizing the dependence on the Galerkin approximation level $N$, remains uniformly bounded with respect to $N$. In particular, we have the following weak(-$\star$) convergence of a subsequence, which we do not relabel, due to the Banach-Alaoglu theorem (see, e.g., \cite[Thm. 3.2.1]{Buhler})
\begin{alignat}{2}
u^{N} &\stackrel{}{\relbar\joinrel\rightharpoonup} u \quad &&\textrm{weakly-$\star$ in } L^\infty \!\left(0,T; \Honefour\right) = L^1\left(0,T;\Honefour^*\right)^*,\\
u_t^{N} &\stackrel{}{\relbar\joinrel\rightharpoonup} u_t \quad &&\textrm{weakly in }\quad L^2\!\left(0,T; \Honethree\right), \\
u_t^{N} &\stackrel{}{\relbar\joinrel\rightharpoonup} u_t \quad &&\textrm{weakly-$\star$ in } L^\infty\!\left(0,T; \Honetwo\right) =  L^1\left(0,T;\Honetwo^*\right)^*
\end{alignat}
with $u \in \calU_\diamondsuit(0,T).$
Using the Rellich--Kondrachov theorem \cite[Thm. 6.3]{adams2003sobolev} and the Aubin--Lions--Simon Lemma~\cite[Thm.~5]{simon1986compact}, there exists a strongly convergent subsequence, again not relabeled, in the following sense
\begin{alignat}{2}
\label{eq:strong_convergence}
u^{N} &\longrightarrow u \quad &&\textrm{strongly in } C([0,T]; \Honetwo).
\end{alignat}

Estimate~\eqref{eq:apriori_linear} remains valid in the Galerkin limit ($N\to \infty$) because of the weak lower semicontinuity of norms.

Attainment of the initial data is shown along the lines of~\cite[Thm.~9.12]{salsa2016partial}. The details are omitted here.
\subsubsection*{Step (iv): Uniqueness of solutions}
Assume that there exist two solutions $u_1, u_2 \in \calU(0,T)$ to~\eqref{eq:main_system_linear}. Then their difference, which we denote $\bar u := u_1-u_2 \in \calU(0,T)$, solves the following system
\[ \label{eq:uniqueness}
\left\{\begin{array}{ll}
\lin_0 \bar u_t + \sum_{j=1}^d\lin_j \bar u_{x_j} - \dam (\mathcal{J} \conv \lapl \bar u) = 0&\mathrm{ on } \qquad (0,T) \times \Omega,
\\
\bar u (0) = 0 &\mathrm{ on } \qquad   \Omega,
\\
\bar u =0 &\mathrm{ on } \qquad (0,T) \times   \partial \Omega. 
\end{array}
\right. 
\]
Testing~\eqref{eq:uniqueness} with $\bar u$ yields
\[ \label{eq:uniquness_tested}
( \lin_0 \bar u_t +  \sum_{j=1}^d \lin_j \bar u_{x_j} - \dam (\mathcal{J} \conv \lapl \bar u ) , \bar u)_2 = 0 .
\]
Integrating \eqref{eq:uniquness_tested} from $0$ to $t \in (0,T)$ and using partial integration Lemma \ref{le:mat_ibp} gives 
\[
\begin{aligned}
\frac 12 \|\lin_0^\half \bar u(t)\|^2_{L^2(\Omega)} + \int_0^t (   \jac \mathcal{J} * \bar u , J_0 \jac \bar u + \jac J_0 \bar u)_2 \ds
= \sum_{j=0}^d \int_0^t (\dx{j} \lin_j \bar u , \bar u )_2 \ds .
\end{aligned}
\]
Using Lemma \ref{lemma:kernel} and Young's inequality, we get 
\[ \label{eq:uniqueness_result}
\begin{aligned}
\frac 12 c_{\lin_0}  &\| \bar u(t)\|^2_{L^2(\Omega)} + C_{\mathcal{J}} \int_0^t \| \jac \mathcal{J} * \bar u \|^2_{L^2(\Omega)} \ds \\ \leq& \int_0^t  \big| \sum_{j=0}^d (\dx{j} M_j \bar u +   \jac J_0 \jac \mathcal{J} * \bar u ,  \bar u) \big| \ds\\
\leq& \left(\sum_{j=0}^d \| \partial_{x_j} M_j \|_{L^\infty(0,T;L^\infty(\Omega))} + \frac{\| \jac J_0 \|_{L^\infty(0,T;L^\infty(\Omega))}}{2 C_{\mathcal{J}}} \right) \int_0^t \| \bar u \|^2_{L^2(\Omega)} \ds \\ & + \int_0^t \frac{C_{\mathcal{J}}}{2} \| \jac \mathcal{J} * \bar u \|^2_{L^2(\Omega)}    \ds .
\end{aligned}
\]
Absorbing the $\mathcal{J}*\bar u$ term and applying Grönwall's inequality yields $\| \bar u(t) \|_{L^2(\Omega)}^2 \leq 0$ for all $t \in (0,T)$, so that $\bar u=0$.
\end{proof}
\subsubsection{Fixed-point argument}\label{sec:fp_arg}
We proceed to prove well-posedness of the nonlinear system~\eqref{eq:main_system} in the special case of Westervelt's assumption ($\mu = \kappa = 0$; recall \eqref{eq:Aj_matrices}).
This is achieved by showing that there exists a $T'>0$ such that the mapping 
\[ \label{eq:fixed_point_mapping}
\begin{aligned}
\calT : \calB_R(0,T') \to \calB_R(0,T'),  \quad  
z = (q,w)^\transpose \mapsto u = (p,v)^\transpose,
\end{aligned}
\]
admits a unique fixed point with \(B_R(0,T) \) defined in \eqref{eq:ball}. 
In \eqref{eq:fixed_point_mapping} the image $\mathcal{T} z$ is defined as the solution of the linearized equation~\eqref{eq:main_system_linear} with 
\[ \label{eq:fp_chooseS}
M_j= S_W(z,x) A_j(z,x), \quad M_0= A_0 S_W(z,x) \text{ and } J_0=S_W(z,x), \quad f=S_W(z,x) g
\]
with \( S_W\) as in \eqref{eq:S} and \(A_0,A_j,g\) as in \eqref{eq:main_system}.
The ball of solutions is defined by
\[ \label{eq:ball}
\begin{aligned}
\calB_R (0,T) := \{ z=(q,w) \in \calU_\diamondsuit(0,T) : \| z \|_{\calU(0,T)} \leq R \text{ with } z(0)=u_0 \\
\| q \|_{L^\infty(0,T;L^\infty(\Omega))} \leq \frac{1 - \| \alpha \|_{L^\infty}}{ \lambda}, \| q_t \|_{L^\infty(0,T;L^\infty(\Omega))} \leq D_0 \}
\end{aligned}
\]
with $T,R, D_0>0$ to be chosen appropriately in the course of the following proof.

A fixed point $z \in \calB_R$ of $\mathcal{T}$ satisfies
\[ \label{eq:fixedpoint_eq}
\left\{\begin{array}{ll}
 S_W A_0 z_t + \sum_{j=1}^d S_W A_j  z_{x_j} - S_W (\mathcal{J} \conv \lapl z) = S_W g&\mathrm{ on } \qquad (0,T) \times \Omega,
\\
z(0) = u_0 &\mathrm{ on } \qquad   \Omega,
\\
z =0 &\mathrm{ on } \qquad (0,T) \times   \partial \Omega. 
\end{array}
\right. 
\]
Since, $S_W$ is positive definite for $z \in \calB_R$, \( S_W^{-1} \) exists and multiplying by $S_W^{-1}$ gives that $z$ is a solution to \eqref{eq:main_system}.

\subsubsection*{Proof of Theorem \ref{th:nonlin-well}}
\begin{proof}

We show that $\mathcal{T}$ is well defined for any $T>0$ by checking the assumptions of Theorem \ref{thm:lin-well}, which guarantees a unique $u \in \calU_\diamondsuit(0,T)$ for every $z \in \calB_R$. From \eqref{eq:fp_chooseS} and \eqref{eq:S} we
see that $M_j$ is symmetric and $J_0$ is diagonal, while $M_0,J_0$ are positive definite if $z \in \calB_R$. Since $\alpha,\beta \in H^{4}(\Omega)$, it is readily verified that $M_j,M_0, J_0 \in \calU(0,T)$ and $f \in \mathcal{V}(0,T)$ for all $z \in \calU(0,T)$. 

Lastly, to fulfill the smallness condition of $\dt J_0$ assumed \eqref{eq:kernel_coff} we see that for any $z=(q,w )\in \mathcal{B}_R$
\[
\begin{aligned}
\| \dt J_0 \|_{L^\infty(0,T;L^\infty(\Omega))} = \frac{\lambda}{1+\| \beta \|_{L^\infty(\Omega)}} \| \dt q \|_{L^\infty(0,T;\Linf)} 
\\
< C_K \inf_{t \in (0,T), \, x \in \Omega} \frac{1+\alpha(x) + \lambda q(t,x)}{1+\beta(x)} = C_K \inf_{t \in (0,T), \, x \in \Omega} \lambda_{min} ( J_0(t,x) )
\end{aligned}
\]
holds, by choosing $$D_0= \lambda^{-1} (1+ \| \beta \|_{L^\infty(\Omega)}) \half C_K \inf_{t \in (0,T), \, x \in \Omega} \frac{1+\alpha(x) + \lambda q(t,x)}{1+\beta(x)} > 0.$$ We note that the infimum is strictly positive, since $q$ is bounded by $\lambda^{-1} (1- \| \alpha \|_{L^\infty(\Omega)})$. 

To apply Banach's fixed point theorem, it remains to ensure the self mapping and contraction property of $\calT$ in a suitable Banach space.
\subsubsection*{Self-mapping property}
Let $R>0$. We prove $u := \calT z \in \calB_R$ if $z \in \calB_R$. First, we show that $\calB_R$ is non-empty.
To this end, consider the constant in time function $c_0: t \mapsto u_0$. Clearly $c_0 \in \calU_\diamondsuit(0,T).$ We choose $r_0$ from the smallness assumption \eqref{eq:smallness_1} small enough so that  $\| c_0 \|_{\calU(0,T)} \leq R$ and that $\| c_0 \|_{L^\infty(0,T;L^\infty(\Omega))} = \| u_0 \|_{L^\infty(\Omega)} < \lambda^{-1} (1- \| \alpha \|_{\Linf} ) $. Automatically, we have $\| \dt c_0 \|_{\Linf} = 0 < D_0. $

To show the self mapping property, we rely on the estimates \eqref{eq:estimate_u_t_H2_v2} and \eqref{eq:estimate_u_t_H3_v2} with \eqref{eq:energy_estimate_H4} in the proof of Theorem~\ref{thm:lin-well} and use smallness of $r_0$ in \eqref{eq:smallness_1}. First, we note that the norm of $u$ in $\calU(0,T)$ can be made sufficiently small if the energy $\mathcal{E}$ and $r_0$ are small. However $\mathcal{E}$ itself can be made sufficiently small if $r_0$ is small.  Thus, from the estimates of the previous proof, we conclude $\| u \|_{\calU(0,T)} \leq R$ if $r_0$ is chosen sufficiently small.

We are left to show that $\| p \|_{L^\infty(0,T;L^\infty(\Omega))} \leq \frac{1- \| \alpha \|}{\lambda}$ and $
\| p_t \|_{L^\infty(0,T;L^\infty(\Omega))}< D_0$, which is achieved by, if needed, further decreasing $r_0$ and exploiting the embedding $  \calU(0,T) \embed W^{1, \infty}(0,T;L^\infty(\Omega))$. 
\subsubsection*{Contraction}

Let $z^{(1)}, z^{(2)} \in \calB_R$ and denote $\bar{z} = z^{(1)} - z^{(2)}$, $u^{(1)} = \calT z^{(1)}$, and $u^{(2)} = \calT z^{(2)}$ with $\bar{u} = u^{(1)} - u^{(2)}$. Then, by definition of $\calT$ we have
\[
\begin{aligned}
    A_0 u_t^{(1)} + \sum_{j=1}^d A_j(z^{(1)}) u^{(1)}_{x_j} - \mathcal{J} * \lapl u^{(1)} &= g, \\
    A_0 u_t^{(2)} + \sum_{j=1}^d A_j(z^{(2)}) u^{(2)}_{x_j} - \mathcal{J} * \lapl u^{(2)} &= g.
\end{aligned}
\]
Subtracting the equations and multiplying by $S_W(z^{(1)})$ yields that
$\bar{u}$ solves
\[ \label{eq:contractive_eq}
\left\{
\begin{array}{ll}
\lin_0(z^{(1)}) \bar{u}_t + \sum_{j=1}^d M_j(z^{(1)}) \bar{u}_{x_j} - J_0(z^{(1)}) (\mathcal{J} * \Delta \bar{u} ) = h & \text{on } (0,T) \times \Omega, \\
\bar{u}(0) = 0 & \text{on } \Omega, \\
\bar{u} = 0 & \text{on } (0,T) \times \partial \Omega
\end{array}
\right.
\]
with $h= S_W(z^{(1)} )\sum_{j=1}^d M_j(\bar{z}) u^{(2)}_{x_j}$. Since $h(t)$ is not in $L^1(0,T;\Honefour) \subset \mathcal{V}(0,T)$,  it is not directly possible to apply Theorem~\ref{thm:lin-well} to get an estimate of $\bar{u}$ depending on $\bar{z}$. Thus, we only show that $\calT$ is contractive with respect to the weaker norm  $L^\infty\!\left(0,T;L^2(\Omega)\right)$. 
Note that this suffices to show the existence of a fixed point as can be seen from Lemma \ref{lemma:banach}. The spaces 
\[
\begin{aligned}
X=\calU_\diamondsuit(0,T) \, \embed \, Y = L^\infty\!\left(0,T;L^2(\Omega)\right) 
\end{aligned}
\] 
and thus satisfy the assumptions from Lemma \ref{lemma:banach}, since \(X\) has a separable predual space. 
By definition $\mathcal{B}_R$ is non-empty and bounded in $\calU_\diamondsuit(0,T)$. Also, $\mathcal{B}_R$ is a ball in $\calU_\diamondsuit(0,T) \cap W^{1,\infty}(0,T;L^\infty(\Omega))$, implying by Banach-Alaoglu \cite[ Thm. 3.2.4]{Buhler} that $\mathcal{B}_R$ is weakly-$\star$ closed in \( X\).

To show contractivity with respect to $Y$, we test~\eqref{eq:contractive_eq} with $\bar u$ to obtain
\[
( \lin_0(z^{(1)})  \bar{u}_t + \sum_{j=1}^d M_j(z^{(1)}) \bar{u}_{x_j} - J_0(z^{(1)}) (\mathcal{J} * \Delta \bar{u} ), \bar{u} )_2 = (h, \bar{u} )_2 .
\]
Performing exactly the same manipulations  as in Step (iv) of the proof of Theorem \ref{thm:lin-well} together with Sobolev embeddings leads to the estimate (\emph{cf.} 
\eqref{eq:uniqueness_result})
\[ \label{eq:contrative_energy_estimate}
\begin{aligned}
\frac 12 c_{\lin_0} \| \bar u(t)\|^2_{L^2(\Omega)} +  \frac{1}{2}C_{\mathcal{J}} \int_0^t \| \jac \mathcal{J} * \bar u \|^2_{L^2(\Omega)} \ds 
\leq C_{4}C_A R   \int_0^t \| \bar u \|^2_{L^2(\Omega)} + (h, \bar u)_2 \ds 
\end{aligned}
\]
with $C_4>0$.
Further, we estimate
\[
\begin{aligned}
\int_0^t (h, \bar u)_2 \ds &\leq \intt \| h \|_{L^2(\Omega)} \| \bar u \|_{L^2(\Omega)} \ds
\\
&\leq \intt 
\| S(z^{(1)}) \|_{\Linf} \sum_{j=1}^d \|  M_j(\bar{z}) \|_{\Ltwo} \| u^{(2)}_{x_j} \|_{\Linf} \| \bar u \|_{\Ltwo} 
\\
&\leq C_A^2 R^2 \half \intt  \|  M_j(\bar z) \|_{\Ltwo}^2 +  \| \bar u \|_{\Ltwo}^2  \ds .
\end{aligned}
\]
Therefore, we obtain by Grönwall's ineqality
\[
\begin{aligned}
\| \bar u \|^2_{L^\infty(0,T;\Ltwo)} \leq C_5  \| \bar z \|^2_{L^2(0,T;\Ltwo)} \leq C_5 T \| \bar z \|^2_{L^\infty(0,T;\Ltwo)}  .
\end{aligned}
\]
Choosing $T>0$ so small that $T < C_5^{-1}$, we obtain contractivity in $L^\infty(0,T;\Ltwo)$.
Thus, using Lemma \ref{lemma:banach} Banach's fixed point theorem guaranties a unique $z \in \calB_R$, solving \eqref{eq:main_system}.
\end{proof}
Using assumption \ref{as:kernel2} we can slightly refine the above theorem relaxing the smallness conditions on $r_0$.
\subsubsection*{Proof of Theorem \ref{th:nonlin-well-refined}}
\begin{proof}
We use the same mapping as in the proof of the previous Theorem \ref{th:nonlin-well}, \emph{cf.} \eqref{eq:fixed_point_mapping}, \eqref{eq:fp_chooseS}, but now use as domain the bigger ball
\[
\begin{aligned}
\calB_R := \{ z=(q,w) \in \calU_\diamondsuit(0,T) : \| z \|_{\calU(0,T)} \leq R \text{ with } z(0)=u_0 \\
\| q \|_{L^\infty(0,T;L^\infty(\Omega))} \leq \frac{1 - \| \alpha \|_{L^\infty}}{ \lambda}\}.
\end{aligned}
\]
The map is shown to be well defined through the same arguments as in the proof of Theorem \ref{th:nonlin-well}.
The remarkable difference is now, that to satisfy the assumption
\[
\begin{aligned}
\| \dt J_0 \|_{L^\infty(0,T;L^\infty(\Omega))}
< C_K \inf_{t \in (0,T), \, x \in \Omega} \lambda_{min} ( J_0(t,x) )
\end{aligned}
\]
we decrease $T$ as needed to ensure Lemma \ref{lemma:kernel2} holds. We note that degeneracy of $T$ is not possible, since $\|  S_W(u) \|_{W^{1, \infty}(0,T;L^\infty)}< \infty$ for $u \in \mathcal{B}(0,T)$. 

We remark that contractivity is shown along the same arguments as in the proof of Theorem \ref{th:nonlin-well}, for $T(R)>0$ chosen small enough.
\subsubsection*{Self-mapping property refined}
We prove $u := \calT z \in \calB_R$ if $z \in \calB_R$. 
We define $R:= \max\{R_1,R_2,R_3,R_4\}$, where $R_i$ are specified below.
First, we show that $\calB_R$ is non-empty.
For this consider the constant in time function $c_0: t \mapsto u_0$. Clearly $c_0 \in \calU_\diamondsuit(0,T).$ Choosing $R_1=\| c_0 \|_{\calU(0,T)}$, we have $\| c_0 \|_{\calU(0,T)} \leq R$. Moreover, $\| c_0 \|_{L^\infty(0,T;L^\infty(\Omega))} = \| p_0 \|_{L^\infty(\Omega)} \le \lambda^{-1} (1- \| \alpha \|_{\Linf} ) $ by \eqref{eq:refine_theorem_assumption}. 

To show the self mapping property for suitably chosen $R$ and $T$ we rely on the estimates \eqref{eq:estimate_u_t_H2_v2} and \eqref{eq:estimate_u_t_H3_v2} with \eqref{eq:energy_estimate_H4} in the proof of Theorem~\ref{thm:lin-well}. First we note that for each used norm in these estimates $\| \cdot \|$ we have $\| M_j,J_0\|\leq C_A R$ with a constant $C_A$ depending on the nonlinearity parameters and $\alpha,\beta$.
With this we now bound the term $\exp{\left( \int_0^T Q(s) \ds \right)} \leq C_{exp}$ independent of $R$ by choosing $T(R)>0$ small enough. Recall that $Q$ was defined in \eqref{eq:Q}. We have
\[
\begin{aligned}
&\int_0^T Q(s) \ds \\
&\leq \sqrt{T} \| M_0 \|_{\calU(0,T)} + T \| M_0 \|_{\calU(0,T)} + \left( T + T \|M_0 \|^2_{\calU(0,T)} \right) \left( \sum_{j=1}^d T \| M_j \|_{\calU(0,T)} + T \| J_0 \|_{\calU(0,T)} \right)  \\
&\leq \sqrt{T} C_A R + T C_A R + (  T+ T C_A^2 R^2 ) (T d C_A R + T C_A R) \\
&\leq \sqrt{T} C_AR  + T  C_A R + T^2 ( 1 + C_A^2 R^2 ) (1+d) C_A R .
\end{aligned}
\]
Thus, choosing $T$ small enough, we obtain an independent bound on $R$.
From estimate \eqref{eq:energy_estimate_H4} we have
\[
\| u \|_{L^\infty(0,T;\Honefour)}^2 \leq C_G C_{exp} \left( C_R \| M_0 \|_{L^\infty(0,T;H^2(\Omega))} \| u_0 \|^2_{\Honefour} + \| f \|_{\mathcal{V}(0,T)}^2 \right) 
\]
and, therefore, with
\[
 \| M_0 \|_{L^\infty(0,T;H^2(\Omega))} \leq \| M_0(0) \|_{H^3(\Omega)} + \sqrt{T} \| \dt M_0 \|_{L^2(0,T;H^3(\Omega))} \leq C_A \| u_0 \|_{\Honefour} + \sqrt{T} C_A R
\]
we obtain an $R$ independent estimate $$\| u \|_{L^\infty(0,T;\Honefour)}^2 \leq R_2^2:= C_G C_{exp} \left( C_R (C_A \| u_0 \|_{\Honefour} + C_A ) \| u_0 \|^2_{\Honefour} + \| f \|_{\mathcal{V}(0,T)}^2 \right) $$ choosing $T$ small enough. Similarly, from estimate \eqref{eq:estimate_u_t_H3_v2} we have using
\[
 \| J_0 \|_{L^\infty(0,T;\Honethree} \leq \| J_0(0) \|_{\Honethree} + \sqrt{T} \| \dt J_0 \|_{L^2(0,T;\Honethree)}
\]
that
\[
\begin{aligned}
\| u_t \|_{L^2(0,T;\Honethree} \leq C_3 (  \| M_0(0) \|_{H^3(\Omega)} + \sqrt{T} \| \dt M_0 \|_{L^2(0,T;H^3(\Omega))} ) (  \| M_0(0) \|_{H^3(\Omega)} \\
+ \sqrt{T} \| \dt M_0 \|_{L^2(0,T;H^3(\Omega))} + d T C_A R  ) \sqrt{R_2} + \| f \|_{L^2(0,T;\Honethree)} \leq R_3,
\end{aligned}
\]
if $T$ is chosen small enough. 
Using \eqref{eq:estimate_u_t_H2_v2}, we have
\[ 
\begin{aligned}
\| u_t \|_{L^\infty(0,T;\Honetwo)} \leq C_2 \|M_0\|_{L^\infty(0,T;H^2(\Omega) )} \Bigl( \| f \|_{L^\infty(0,T;\Honetwo)} +  \bigl( \| J_0 \|_{L^\infty(0,T;H^2(\Omega)} \\
+ \sum_{j=1}^d \|  M_j \|_{L^\infty(0,T;H^2(\Omega))} \bigr)  \sqrt{ R_2}  \Bigr) \leq R_4,
\end{aligned}
\]
if $T$ is chosen small enough. Again $R_4$ can be chosen independent of $R$ by similar arguments as before.
Putting everything together, we obtain $\| u \|_{\calU(0,T)} \leq R$.

We are left to show that $\| p \|_{L^\infty(0,T;L^\infty(\Omega))} \leq \frac{1- \| \alpha \|_{L^\infty(\Omega)}}{\lambda}$. By \eqref{eq:refine_theorem_assumption} we obtain
\[
\begin{aligned}
\| p \|_{L^\infty(0,T;L^\infty(\Omega))} \leq & \| p_0 \|_{L^\infty(\Omega)}+ \sqrt{T} \| \dt p \|_{L^2(0,T;L^\infty(\Omega))})
\\
  \leq & \frac{1- \| \alpha \|_{L^\infty(\Omega)}}{2\lambda} + \sqrt{T} C \| \dt p \|_{L^2(0,T;H^2(\Omega))})\\
   \leq &  \frac{1- \| \alpha \|_{L^\infty(\Omega)}}{2\lambda}+ \sqrt{T} C R,
\end{aligned}
\]
where $C$ is the embedding constant $H^2(\Omega) \embed L^\infty(\Omega)$. Thus:
\[\| p \|_{L^\infty(0,T;L^\infty(\Omega))}\leq   \frac{1- \| \alpha \|_{L^\infty(\Omega)}}{\lambda}
\]
by choosing $T(R)$ small enough.
\end{proof}
\subsection{Inviscid Kuznetsov-type case} \label{se:well-posedness_Kuznetsov}
In this section, we investigate the case of $\mathcal{J}=0$, for whose treatment, we rely on the existing theory for quasilinear symmetric hyperbolic systems on $\R^d$.
To prove Theorem \ref{th:unbounded_domain}, we rely on \cite[Th.~II]{Kato1975}. More recent treatments of this theory can be found in, e.g. \cite[Th.~6.6.1]{Rauch2012} and \cite[Ch.~16, Th.~2.3]{Taylor2011}.

\begin{proof}[Proof of Theorem \ref{th:unbounded_domain}]
Based on assumption $\vartheta = \kappa$, we transform \eqref{eq:main_system} into a quasilinear symmetric hyperbolic system to apply \cite[Theorem II]{Kato1975}.
For this, we define the symmetric matrix
\begin{equation}
S(u) = S(p,v) = 
\begin{pmatrix}
1 + \beta + \mu p & \kappa v^\top \\
\kappa v & (1 + \alpha + \lambda p) \id_{\R^d}
\end{pmatrix} \in \R^{(1+d) \times (1+d)}
\end{equation}
with eigenvalues
\begin{equation}
\begin{aligned}
\sigma_{1,2} &= \frac{ q_1 + q_2 \pm \sqrt{ (q_1 - q_2)^2  + 4 | \kappa v |^2} }{2}, \quad
\sigma_3 &= q_1 ,
\end{aligned}
\end{equation}
where $q_1 = 1 + \alpha + \lambda p$, $q_2 = 1 + \beta + \mu p $ and $\sigma_1, \sigma_2$ have multiplicity $1$, while $\sigma_3$ has multiplicity $d-1$.
Positivity of the eigenvalues is obtained, for example, if
\begin{equation} \label{eq:assump_eigen}
| \alpha + \lambda p | < \frac{1}{8} \quad 
| \beta + \mu p | < \frac{1}{8} \quad
|\kappa v|^2 < \frac{1}{2},
\end{equation}
since then we observe that $\sigma_1$ and $\sigma_3$ are positive. 
For $\sigma_2$ we note that
$$ 
\left( 2 + \alpha + \beta + (\mu + \lambda )p \right)^2  > \left( \alpha - \beta + ( \lambda - \mu ) p \right)^2  + 4 | \kappa v |^2 
$$ 
leads to
$$
1 + 2 \left(\alpha + \beta + (\mu + \lambda) p \right) >  | \kappa v |^2 ,
$$
but we have
$ | 2 (\alpha + \beta + (\mu + \lambda) p ) | < \frac{1}{2}$ so that the inequality holds true if $ | \kappa v |^2 < \frac{1}{2}$.
Note that
\begin{align}
B_j(p,v) &:= S(p,v)A_j(p,v) \\
&=
\begin{pmatrix}
q_2 \kappa v_j + \kappa v^\top q_2 e_j 
& q_2 q_1 e_j^\top + \kappa v^\top ( \kappa v  e_j^\top)  
\\
\kappa v \kappa v_j+ q_1 q_2 e_j 
& \kappa v q_1 e_j^\top + q_1 \kappa e_j v^\top   
\end{pmatrix}
\\ &=  
\begin{pmatrix}
2 \kappa q_2 v_j & q_2 q_1 e_j^\top + \kappa^2 v_j v^\top  \\
\kappa^2 v v_j + q_2 q_1 e_j & \kappa q_1 (  v e_j^\top +  e_j v^\top )
\end{pmatrix}
\end{align}
are symmetric. Thus, $S(u)$ symmetrizes \eqref{eq:main_system} and for certain choices of $\alpha, \beta$ and $u$ small enough we obtain positive eigenvalues of $S$.

We are now left to verify the given Assumptions (4.2) to (4.9) in \cite[Theorem II]{Kato1975}. Assumptions (4.2) to (4.6) follow, since $S(u), A_j(u)$ depend smoothly on $u$ and $f \in C([0,T];H^s(\R^d) )$. Since $B_j(u)$ are symmetric for any $u$, assumption (4.7) also holds. In order to fulfill assumption (4.8), we need to show that the eigenvalues of $S(u)$ remain positive for all $u \in U_0$. We thus check the conditions \eqref{eq:assump_eigen}, that yield
\begin{equation}
| \alpha + \lambda p| \leq | \alpha | + | \lambda | | p | \leq r_\alpha + | \lambda | C_{L^\infty \embed H^s} \| p \|_{H^s} \leq r_\alpha + | \lambda | r_0,
\end{equation}
\begin{equation}
| \beta + \mu p| \leq | \beta | + | \mu | | p | \leq r_\beta + | \mu | C_{L^\infty \embed H^s} \| p \|_{H^s} \leq r_\beta + | \mu | r_0,
\end{equation}
\begin{equation}
| \kappa v | \leq | \kappa |  C_{L^\infty \embed H^s} \| v\|_{H^s} \leq | \kappa | r_0
\end{equation}
obtaining positivity of the eigenvalues by the choices of $r_\alpha, r_\beta, r_0$.
\end{proof}
\begin{remark}
The bounds on $r_\alpha, r_\beta, r_0$ used in the theorem are by no means sharp. However, since $S$ is not symmetric for all $p,v$ one is forced to impose some bounds on the initial conditions and \(\alpha,\beta\).
\end{remark}
%
%
%
%
%
\section{Conclusion}
In this work, we establish well-posedness for a first-order formulation of fractionally damped nonlinear acoustics. Since our analysis does not rely on regularity from damping terms, we symmetrize system~\eqref{eq:main_system} to apply energy estimates that exploit the resulting hyperbolic structure. The nonlinearities introduced by this symmetrization in the non-local damping terms are handled by carefully designed nonlinear coercivity lemmas. For the Westervelt-type system on bounded domains, our results hold locally in time for small data assuming completely monotone memory kernels, while for the Kuznetsov-type system we prove well-posedness in the inviscid case in $\R^d$.

Future work will be concerned with relaxing the assumptions underlying Lemma~\ref{lemma:kernel}.
One direction is to weaken Assumption~\ref{as:kernel}, which is currently formulated for completely monotone kernels~\eqref{def::complete_monot}. We expect that the analysis of \eqref{eq:main_system} can be extended to a broader class of completely positive kernels \cite[Def.~1]{clement1981asymptotic}, provided that the associated resolvents exhibit sufficient regularity (possibly after suitable regularization) to establish Lemma~\ref{lemma:abstractkernel}.

A second direction concerns the diagonal matrix assumption in Lemma~\ref{lemma:kernel}. It remains open whether this assumption is necessary for singular kernels. A positive answer would enable well-posedness for the Kuznetsov-type system using the methods developed here.

Another direction is to study the limiting behavior of \eqref{eq:main_system} as the memory kernels vanish $\mathcal{J} \to 0$. 

Additionally, the study of system \eqref{eq:main_system} with more general boundary conditions is of interest. In the case $\mathcal{J}=\delta_0$, well-posedness of inhomogeneous Dirichlet and slip boundary conditions is shown in \cite{lehner:2026}.

In summary, this work takes a first step toward first-order fractional-in-time models for nonlinear acoustics. It lays a foundation for future research on more detailed mathematical analysis and provides a solid basis for efficient numerical schemes, with potential applications such as parameter identification for \eqref{eq:main_system}.

\section*{Acknowledgment}
P. Lehner was funded in whole or in part by the Austrian Science Fund (FWF)[10.55776/P36318]. 
The work of M. Meliani is supported by ERC Synergy Grant PSINumScat - 101167139.

\bibliographystyle{plain}
\bibliography{lit}

\clearpage
\appendix
\section{Derivation of physical model}\label{se:derivation}
We consider a three dimensional physical space. The governing equations for the model are given by the Navier–Stokes–Fourier system
\begin{equation}\label{eq:NSFsystem}
\left\{
\begin{aligned}
\dt \den 
+ \divv(\den \vel) 
&= 0, 
\\
\den \bigl(\dt \vel + (\vel \cdot \grad)\vel \bigr) 
&= \Divv{\stress} + \den \force,
\\
\den \tem \bigl(\dt \ent + \vel \cdot \grad \ent\bigr) 
&= -\divv{\heat} + \work,
\end{aligned}
\right.
\end{equation}
see \cite[p.~576]{Pierce:2019} and \cite[Ch.~2]{enflo2002}, whose equations correspond to the conservation of mass, momentum, and energy, respectively. The constitutive relations considered in this work are assumed to be
\begin{equation} \label{eq:const_relation}
\left\{
\begin{aligned}
\stress 
&= -\pre \id 
+ 2 \shearK \conv \!\left( \dt{\strain} - \tfrac{1}{3} \tr(\dt \strain) \id \right) 
+ 2 \bulkK \conv \tr(\dt \strain) \id,
\\
\dt \strain 
&= \tfrac{1}{2} \bigl( \jac \vel + \jac^\transpose \vel \bigr),
\\
\heat 
&= - \heatK \conv \grad \tem,
\\
\work 
&= \bigl( \stress + \pre \id \bigr) \ccdot \jac \vel,
\end{aligned}
\right.
\end{equation}
see \cite[Eq.~(30)]{Holm2013}, \cite{Prieur2012}.
In \eqref{eq:NSFsystem} and \eqref{eq:const_relation} $\den$ denotes the mass density, $\vel$ the velocity field, $\tem$ the temperature distribution, $\ent$ the specific entropy, $\strain$ the (linearized) strain tensor, $\force$ the external force density, $\heat$ the heat flux, and $\work$ the dissipative work due to viscous friction. The (linearized) stress tensor $\stress$ is decomposed into an isotropic pressure part $\pre$, a traceless deviatoric component, and a purely dilatational part, the latter corresponding, respectively, to shear and bulk viscosity effects. The time-dependent kernels $\heatK$, $\shearK$, and $\bulkK$ are assumed to be positive and can be thought of as diagonal matrices that can be identified with a vector containing the diagonal entries. 

By choosing $\shearK = \mu_S \delta_0$,  $\bulkK = \mu_B \delta_0$, $\heatK = \kappa \delta_0$, where $\delta_0$ denotes the Dirac distribution centered at $t = 0$ and $\mu_S$ shear viscosity, $\mu_B$ bulk viscosity, $\kappa$ heat conductivity coefficients, one recovers the standard constitutive relations used in nonlinear acoustics, see \cite{enflo2002}. 
The following derivation closely follows the arguments presented in \cite[Sec. 2]{QuadraticWave}, with the primary difference being that their analysis is carried out in a unitless framework and the assumption $\shearK = \mu_S \delta_0$,  $\bulkK = \mu_B \delta_0$, $\heatK = \kappa \delta_0$. For completeness and clarity, we reproduce the derivation here. 
\subsection{Modeling assumptions}
We assume the existence of an in space and time constant state
\begin{equation} \label{eq:constant_states}
\bigl( \den_0, \vel_0, \ent_0 ,\tem_0, p_0 \bigr) \quad \text{ with }\vel_0 = 0
\end{equation}
of the homogeneous, quiescent medium and make use of the small amplitude approximation
\begin{equation} \label{eq:SAA}
\den = \den_0 + \mach \de, \quad 
\vel =		    \mach \ve, \quad
\ent = \ent_0 + \mach \en, \quad
\tem = \tem_0 + \mach \te, \quad
p = p_0 + \mach \pr
\end{equation}
with Mach number $\mach \ll 1$; see  \cite{DekkersRozanova2020, Tani:2017}. In particular, all derivatives of $\den_0, \ent_0, \tem_0, p_0$ vanish. Additionally, we assume
\begin{equation} \label{eq:SAAA}
\shearK, \, \bulkK, \, \heatK \in \bigO(\mach) \quad 
\force, \, \curl \vel \in \bigO(\mach^2) \quad 
\curl \paren*{\curl \vel} \in \bigO(\mach^3) .
\end{equation}
Neglecting certain nonlinearities, one obtains additionally the linear relation
\begin{equation} \label{eq:wester_assumption}
\mach \den_0^2 c_0^2 \abs{\ve}^2 = \mach \pr^2 + \bigO(\mach^2) 
\end{equation}
with $c_0$ speed of sound.  Assumption~\eqref{eq:wester_assumption} is used to derive the Westervelt equation~ \cite{Westervelt1963Parametric}.
\subsection{Navier-Stokes-Fourier system up to $\bigO(\mach^3)$}
In the derivation, we make use of the substitution corollary, a procedure known in nonlinear acoustics as Blackstock’s scheme, see \cite{Blackstock:63}.
This approach allows substituting linear relations in $\mach$ into terms of quadratic order in $\mach$, since no information is lost at the relevant order of approximation.
\subsubsection{Momentum equation}
Plugging the constitutive relations~\eqref{eq:const_relation} into the right hand side of the momentum equation, using~\eqref{eq:SAA},~\eqref{eq:SAAA}, and $\tr( \dt \strain)=\divv \vel$, we arrive at
\begin{equation} \label{eq:momentum_rhs}
\begin{aligned}
\Divv{\stress} + \den \force 
&= - \grad \pre + \shearK \conv \paren*{\grad (\divv \vel) + \laplvec \vel - \tfrac{2}{3} \grad (\divv \vel )  } \\ &\quad + 
\bulkK \conv \grad (\divv \vel) + \den_0 \force + \bigO(\mach^3) \\
&= - \grad \pre +  \shearK \conv \laplvec \vel + \paren*{ \tfrac{1}{3} \shearK +  \bulkK  } \conv \grad (\divv \vel)  + \den_0 \force + \bigO(\mach^3) \\
&= - \grad \pre +  \mach \paren*{ \tfrac{4}{3} \shearK + \bulkK  } \conv \laplvec \ve + \den_0 \force + \bigO(\mach^3),
\end{aligned}
\end{equation}
where in the last line we use the vector calculus identity $ \grad (\divv \vel ) = \laplvec \vel + \curl \paren*{ \curl \vel }.$ A Taylor expansion of $\pre(\den, \ent)$ around $(\den_0, \ent_0)$ yields
\begin{equation} \label{eq:pressure_taylor}
\begin{aligned}
\pre &=  \pre(\den_0, \ent_0) +  \paren*{ \partial_{\den}  \pre }_0 ( \den - \den_0 ) +  \paren*{ \partial_{\ent} \pre}_0 ( \ent - \ent_0 ) 
+ \half  \paren*{ \partial^2_{\den} \pre}_0 (\den-\den_0)^2 \\ 
&\quad+ \half  \paren*{\partial^2_{\ent} \pre }_0 (\ent-\ent_0)^2 +  \paren*{ \partial^2_{\den, \ent} \pre}_0 (\den - \den_0 ) (\ent- \ent_0 )  + \dots 
\\
&= \pre_0 + c_0^2  \mach \de  +  d_0^2  \mach \en 
+ \half   E_0 (\mach \de )^2   + \half F_0 ( \mach \en)^2 +  G_0 \mach^2 \de \en  + \bigO(\mach^3)
\end{aligned}
\end{equation}
with $\pre_0=\pre(\den_0,\ent_0)$, $c_0^2 = \paren*{ \partial_{\den}  \pre }_0:= \partial_{\den}  \pre(\den_0,\ent_0)$, $ d_0^2 = \paren*{ \partial_{\ent} \pre}_0 := \partial_{\ent}  \pre(\den_0,\ent_0) $, $E_0 = \paren*{ \partial^2_{\den} \pre}_0$, $F_0 = \paren*{ \partial^2_{\ent} \pre}_0$, and $ G_0 = \paren*{ \partial^2_{\den, \ent} \pre}_0.$ Note that $E_0=\frac{B}{A} c_0^2 \rho_0^{-1}$, where $\frac{B}{A}$ is the nonlinearity parameter used in nonlinear acoustics.
In particular, we get
\begin{equation} \label{eq:pressure_linear}
\mach \de =	\mach c_0^{-2}\pr - c_0^{-2} d_0^2 \mach \en   + \bigO(\mach^2) .
\end{equation}
For the left hand side of the momentum equation, we have using the vector calculus identity $ \paren*{ \vel \cdot \grad } \vel = \half \grad \paren*{ \abs{\vel}^2} + \paren*{ \curl \vel } \times \vel $ and approximations~\eqref{eq:SAA},~\eqref{eq:SAAA}
\begin{equation} \label{eq:momentum_lhs}
\den \bigl( \dt \vel + (\vel \cdot \grad)\vel \bigr) = \mach  \den_0 \dt \ve  + \mach^2 \paren*{ \de \dt \ve + \half \den_0 \grad \paren*{ \abs{\ve}^2}  } + \bigO(\mach^3) .
\end{equation}
Using relation~\eqref{eq:pressure_linear} and $\mach \dt \ve = - \mach \den_0^{-1} \grad \pr + \bigO(\mach^2) $, we obtain by substitution corollary
\begin{equation} \label{eq:momentum_subs}
\mach^2 \de \dt \ve 
= \mach^2 c_0^{-2} \bigl(   \pr -  d_0^2 \en \bigr) ( - \den_0^{-1} \grad \pr) + \bigO(\mach^3) .
\end{equation}
Combining~\eqref{eq:momentum_rhs},~\eqref{eq:momentum_lhs},~\eqref{eq:momentum_subs} we have, neglecting $\bigO(\mach^3)$ terms,
\begin{equation} \label{eq:momentum_final}
\begin{aligned}
\mach   \den_0 \dt \ve  +  \mach^2 \frac{\den_0}{2} \grad \paren*{ \abs{\ve}^2}  -  \mach  \viscK  \conv \laplvec \ve  
= \mach \paren*{  \mach  c_0^{-2} \den_0^{-1}  \pr - \mach S_\vel   - 1 } \grad \pr + \den_0 \force
\end{aligned}
\end{equation}
with $\viscK:=\frac{4}{3} \shearK + \bulkK$ and $S_\vel:=  c_0^{-2} \den_0^{-1}  d_0^2 \en_C $, where $\en_C$ comes from~\eqref{eq:entropy_constant}, see below. Using assumption~\eqref{eq:wester_assumption} one even has by substitution corollary
\begin{equation} \label{eq:wester_momentum}
\mach^2 \frac{\den_0}{2} \grad \paren*{ \abs{\ve}^2} = \mach^2  \frac{\den_0}{2} \grad \paren*{ \den_0^{-2} c_0^{-2} \pr^2} = \mach^2 \den_0^{-1} c_0^{-2} \pr \grad \pr + \bigO(\mach^3),
\end{equation}
resulting in a cancellation of the nonlinear terms in \eqref{eq:momentum_final}.
\subsubsection{Continuum equation}
In order to approximate the continuum equation, we first, using~\eqref{eq:SAA}, simplify the left hand side of the entropy equation to
\begin{equation} \label{eq:entropy_lhs}
\begin{aligned}
\den \tem \bigl(\dt \ent + \vel \cdot \grad \ent \bigr) &= \mach \den_0 \tem_0 \dt \en+ \mach^2 \paren*{ \den_0 \te \dt \en 
+ \de \tem_0 \dt \en + \den_0 \tem_0 \ve \cdot \grad \en } \! + \bigO(\mach^3)
\end{aligned} 
\end{equation}
and the right hand side, using~\eqref{eq:const_relation},~\eqref{eq:SAA},~\eqref{eq:SAAA} to
\begin{equation} \label{eq:entropy_rhs}
\begin{aligned}
-\divv{\heat} + \work 
&= \heatK \conv \lapl \tem 
+ \paren*{ 2 \shearK \conv \paren*{ \dt{\strain} - \tfrac{1}{3} \tr(\dt \strain) \id } + 2 \bulkK \conv \tr(\dt \strain) \id } \ccdot \jac \vel \\
&= \mach \heatK \conv \lapl \te + \bigO(\mach^3) .
\end{aligned}
\end{equation}
Combining~\eqref{eq:entropy_lhs} and~\eqref{eq:entropy_rhs}, we have
\begin{equation} \label{eq:entropy_linear}
\mach \den_0 \tem_0 \dt \en = 0 + \bigO(\mach^2),
\end{equation}
which also leads to
\begin{equation} \label{eq:entropy_constant}
\mach \en = \mach \en_C + \bigO(\mach^2)
\end{equation}
with $\en_C$ constant in time. A Taylor expansion of $\tem(\den, \ent)$ around $(\den_0, \ent_0)$ yields
\begin{equation}
\tem = \tem_0 + \paren*{ \partial_{\den}  \tem }_0 ( \den - \den_0) +  \paren*{ \partial_{\ent} \tem}_0 (\ent - \ent_0) + \dots 
\end{equation}
so that 
\begin{equation} \label{eq:temp_linear}
\mach \lapl \te =  \mach a_0^2 \lapl \de + \mach b_0^2 \lapl \en + \bigO(\mach^2) 
\end{equation}
with $a_0^2= \paren*{ \partial_{\den}  \tem }_0$ and $b_0^2 =  \paren*{ \partial_{\ent} \tem}_0$. Therefore, using substitutions~\eqref{eq:entropy_linear} and~\eqref{eq:temp_linear} in~\eqref{eq:entropy_lhs} we get, neglecting $\bigO(\mach^3)$ terms,
\begin{equation} \label{eq:entr_sim}
\den_0 \tem_0 \bigl( \mach  \dt \en + \mach^2 \ve \cdot \grad \en \bigr) = \mach \heatK \conv \paren*{  a_0^2 \lapl \de +  b_0^2 \lapl \en} .
\end{equation}
Finally, using substitutions~\eqref{eq:pressure_linear} and~\eqref{eq:entropy_constant} in~\eqref{eq:entr_sim} results in 
\begin{equation} \label{eq:entropy_final}
\mach  \dt \en + \mach^2 \ve \cdot \grad \en_C  = \mach \lambda_0  \heatK \conv   \lapl \pr  + \mach S_\lapl + \bigO(\mach^3)
\end{equation}
with $\lambda_0 := \den_0^{-1} \tem_0^{-1}  a_0^2 c_0^{-2} $ and
$S_\lapl:= \den_0^{-1} \tem_0^{-1}  \paren*{  b_0^2 - c_0^{-2} d_0^2 a_0^2 } \heatK \conv \lapl \en_C $.

The continuum equation using~\eqref{eq:SAA} is
\begin{equation} \label{eq:cont_simple}
\mach \bigl( \dt \de + \den_0 \divv{ \ve } \bigr) + \mach^2 \divv{\bigl( \de \ve \bigr)} = 0 .
\end{equation}
Moreover, using substitutions~\eqref{eq:pressure_linear},~\eqref{eq:entropy_constant} yields
\begin{equation} \label{eq:cont_sub}
\mach^2 \divv{\paren*{ \de \ve}} = \mach^2 \divv \Bigl( c_0^{-2} \bigl(   \pre - d_0^2 \en \bigr) \ve  \Bigr) + \bigO(\mach^3) .
\end{equation}
Differentiating the Taylor series of $\pre$ ~\eqref{eq:pressure_taylor} in time, using substitutions~\eqref{eq:entropy_linear},~\eqref{eq:entropy_constant}, $\mach \dt \de = - \mach \den_0 \divv \ve + \bigO(\mach^2)$, and using~\eqref{eq:entropy_final} gives
\begin{equation} \label{eq:pressure_time}
\begin{aligned}
\mach \dt \pr &= c_0^2  \mach \dt \de  +  d_0^2 \mach \dt \en 
+   E_0 \mach^2 \de \dt \de 
+   G_0 \mach^2 \en \dt \de  + \bigO(\mach^3)  
\\
&= c_0^2  \mach \dt \de  +  d_0^2  \mach \dt \en 
+ \bigl(  E_0 \mach^2 \de  +   G_0 \mach^2 \en \bigr) (-\den_0 \divv \ve)  + \bigO(\mach^3) 
\\
&= c_0^2  \mach \dt \de 
+ \mach d_0^2  \lambda_0 \heatK \conv   \lapl \pr  
+ \mach d_0^2 S_\lapl 
- \mach^2 d_0^2 \ve \cdot \grad \en_C    \\
&-   E_0 \mach^2  c_0^{-2} ( \pr  
- d_0^2  \mach^2 \en_C  ) \den_0 \divv \ve 
- G_0 \mach^2 \en_C \den_0 \divv \ve 
+ \bigO(\mach^3) .
\end{aligned}
\end{equation}
Combining~\eqref{eq:cont_simple},~\eqref{eq:cont_sub}, \eqref{eq:pressure_time} the continuum equation neglecting $\bigO(\mach^3)$ terms is
\begin{equation} \label{eq:cont_final}
\begin{aligned}
\mach \dt \pr &+ \mach \paren*{ \den_0 c_0^2  + \mach S_\pre 	+  \mach  \den_0 E_0  c_0^{-2} \pr } \divv \ve \\
&+ \mach^2 \divv ( \pr \ve ) - \mach d_0^2 \lambda_0 \heatK \conv \lapl \pr = - \mach d_0^2 S_\lapl 
\end{aligned}
\end{equation}
with $S_\pre:= ( E_0 c_0^{-2} d_0^2 \den_0    +  G_0  \den_0 - d_0^2 ) \en_C$. Note that $\grad \en \cdot \ve$ terms cancel out. 
\subsubsection{Summary}
Writing~\eqref{eq:momentum_final} and~\eqref{eq:cont_final} in terms of $\pre' := \mach \pr$ and $\vel':= \mach \den_0 c_0 \ve $, we obtain an approximation of~\eqref{eq:NSFsystem}, as
\begin{equation} \label{eq:physical_system}
\begin{aligned}
\dt \pre' +  \paren*{ c_0  + \mach r_0 S_\pre 	+ 2 E_0  c_0^{-1} \pre' } \divv \vel' + r_0 \divv ( \pre' \vel' ) - d_0^2 \lambda_0 \heatK \conv \lapl \pre'  &=  - \mach d_0^2 S_\lapl 	\\
\dt \vel' +   \paren*{ 1 + \mach S_\vel  -  c_0^{-2} \den_0^{-1}  \pre'     } \grad \pre' + \frac{1}{2 \den_0 c_0^2} \grad \paren*{ \abs{\vel'}^2}  - r_0 \viscK  \conv \laplvec \vel'  &= 
\den_0 \force 
\end{aligned}
\end{equation}
with $r_0=\den_0^{-1} c_0^{-1}$.
Assuming~\eqref{eq:wester_assumption}, which yields~\eqref{eq:wester_momentum}, gives a linear momentum equation
\begin{equation} \label{eq:physical_westervelt}
\dt \vel' +   \paren*{ 1 + \mach S_\vel     } \grad \pre' -  r_0 \viscK  \conv \laplvec \vel'  = 
\den_0 \force .
\end{equation}
To obtain~\eqref{eq:main_system} from~\eqref{eq:physical_system}, we set 
\begin{equation}
\begin{aligned} \label{eq:physical_parameters}
A_0 = \diagmatrix(c_0^{-1}, \, 1, \dots, 1), \quad 
\mathcal{J} = (c_0^{-1} d_0^2 \lambda_0 \mathcal{K}, \, \den^{-1}_0 c_0^{-1} \mathcal{V} )^\transpose, \\ \quad 
f=(-M c_0^{-1} d_0^2 S_{\lapl}, \, \den_0 \force)^\transpose \hphantom{ppppppppppp}
\end{aligned}
\end{equation}
and $\alpha= M \den_0^{-1} c_0^{-1} S_p$, $\beta = M S_v$, $\vartheta = \mu = \kappa = c_0^{-2} \den_0^{-1}$, $\lambda = 2 E_0 c_0^{-1} + c_0^{-2} \den_0^{-1}$, where we make use of the product formulas for $\grad$.
Observe that \( \lambda \neq \vartheta \), such that the nonlinear system cannot be written in a conservative form; the system would be conservative if, e.g., $E_0 = 0$ in \eqref{eq:physical_system}.

\section{Boundary information} \label{se:appendix_galerkin}
In this section, we show that the two equations \eqref{ineq:bnd_value_convection_term} and \eqref{ineq:bnd_value_convection_term_lapl} are true in the Galerkin setting of the proof of Theorem \eqref{thm:lin-well}.

To show~\eqref{ineq:bnd_value_convection_term} holds, notice that for all $\phi^N \in U_N$~\eqref{eq:galerkin_projection} gives
\[ \label{eq:gal_test} 
\sum_{j=1}^d( \lin_j u_{x_j}^N, \phi^N)_2 = (- \lin_0 \dt u^N + \dam \mathcal{J} \conv \lapl u^N + f,\phi^N)_2 
\]
and since $-\lapl \phi^N \in U_N$ 
\[ \label{eq:gal_lapl_test} 
\sum_{j=1}^d ( \lin_j u_{x_j}^N, -\Delta \phi^N)_2 = (- \lin_0 \dt u^N + \dam \mathcal{J} \conv \lapl u^N + f,-\Delta \phi^N)_2.
\]
Integrating~\eqref{eq:gal_lapl_test} by parts with Lemma \ref{le:lapl_ibp}, using~\eqref{eq:zero_bnd_properties},~\eqref{eq:zero_bnd_properties_extended}, as well as $f|_{\partial \Omega}=0$, we find 
\begin{equation} 
\label{eq:gal_lapl_test_pi} 
\begin{multlined}
\sum_{j=1}^d (-\Delta (\lin_j u_{x_j}^N), \phi^N)_2 - \sum_{j=1}^d\int_{\partial \Omega}  \lin_j u_{x_j}^N  \cdot \jac \phi^N n  \dS \\ = ( \Delta (\lin_0 \dt u^N ) - \Delta ( \dam \mathcal{J} \conv \lapl u^N) - \Delta f,\phi^N)_2.
\end{multlined}
\end{equation}
From~\eqref{eq:gal_test} the residual $r^N:=\lin_0 u_t^N + \sum_{j=1}^d \lin_j u_{x_j}^N  - \dam \mathcal{J} \conv \Delta u^N - f$ of~\eqref{eq:galerkin_projection} 
satisfies
\begin{equation}
\label{eq:residual_orthogonal}
(r^N, \phi^N)_{2} = 0 \quad \forall \phi^N \in U_N
\end{equation} 
Since $r^N \in L^2\left(0,T;L^2(\Omega)\right)$ it follows from~\eqref{eq:residual_orthogonal} that $r^N= \sum_{K>N}^\infty r_K w_K$. Therefore, we conclude that, due to orthogonality, 
\[
(\Delta r^N, \phi^N)_2 = (\sum_{K > N}^\infty \lambda_K r_K w_K, \phi^N)_2 = 0.
\]
It follows that
\[ \label{eq:gal_laplace_projection}
(-\Delta \sum_{j=1}^d \lin_j u_{x_j}^N, \phi^N)_2 = (\Delta (\lin_0 \dt u^N ) - \Delta ( \dam \mathcal{J} \conv \lapl u^N )  -\Delta f, \phi^N)_2,
\]
which yields, when 
combining~\eqref{eq:gal_lapl_test_pi} and~\eqref{eq:gal_laplace_projection} to~\eqref{ineq:bnd_value_convection_term}. 

Similarly, one can show that \eqref{ineq:bnd_value_convection_term_lapl}
for all $\phi^N \in U_N$ using that $\Delta f|_{\partial \Omega} = 0$. 
Since $\lapl^2 \phi^N \in U_N$ we have
\[ \label{eq:gal_lapl^2_test} 
\sum_{j=1}^d ( \lin_j u_{x_j}^N, \Delta^2 \phi^N)_2 = (- \lin_0 \dt u^N + \dam \mathcal{J} \conv \lapl u^N + f,\Delta^2 \phi^N)_2.
\]
Integration by part gives using \eqref{ineq:bnd_value_convection_term}
\begin{equation} 
\label{eq:gal_lapl^2_test_pi} 
\begin{multlined}
\sum_{j=1}^d (-\Delta (\lin_j u_{x_j}^N), \Delta \phi^N)_2  = ( \Delta (\lin_0 \dt u^N ) - \Delta ( \dam \mathcal{J} \conv \lapl u^N) - \Delta f,\Delta \phi^N)_2
\end{multlined}
\end{equation}
and doing another integration by parts
\begin{equation} 
\label{eq:gal_lapl^2_test_pi2} 
\begin{multlined}
\sum_{j=1}^d (-\Delta^2 (\lin_j u_{x_j}^N), \phi^N)_2 + 
\int_{\partial \Omega} \lapl ( \lin_0 u_t^N + \sum_{j=1}^d \lin_j u_{x_j}^N - \dam \mathcal{J} \conv \lapl u^N) \cdot \jac \phi^N n \dS \\
= ( \Delta^2 (\lin_0 \dt u^N ) - \Delta^2 ( \dam \mathcal{J} \conv \lapl u^N) - \Delta^2 f, \phi^N)_2
\end{multlined}
\end{equation}
Since we also know that $(\Delta^2 r^N, \phi^N) = 0$ we get that the boundary terms vanish, once summed up. However, we need to ensure that $\Delta^2 r^N \in L^2(\Omega)$. This follows from the Banach algebra $H^4$ for $d \leq 3$ and elliptic regularity.
%
%
%
%
\section{Proof of Lemmas~\ref{lemma:kernel} and ~\ref{lemma:kernel2}: Nonlinear coercivity} \label{se:kernel}
The purpose of this section is to prove inequality \eqref{eq:kernel_assump} as stated in Lemma \ref{lemma:kernel}, which is essential in our analysis. Let $T>0$ and let $\mathcal{J} \in L^1_{\mathrm{loc}}(0,T)$ be a positive kernel that is completely monotone on $(0,T)$.
Then $\mathcal{J}$ admits a resolvent of the first kind, by which we mean that there exists a finite Radon measure $\mathfrak{r} \in \mathcal{M}_{\textup{loc}}(0,T)$ such that
\[
(\mathfrak{r} * \mathcal{J})(t)
= \int_0^t \mathcal{J}(t-s)\, d\mathfrak{r}(s)
= 1
\qquad \text{for a.e.\ } t \in (0,T)
\]
see \cite[Chapter~5, Definition 5.1]{Gripenberg1990}.
Moreover, by \cite[Chapter~5, Theorem~5.4,5.5]{Gripenberg1990}, the resolvent $\mathfrak{r}$ admits the decomposition
\begin{equation}\label{resolvent_form}
\mathfrak{r} = a\, \delta_0 + k,
\end{equation}
where $a\geq0$, $\delta_0$ denotes the Dirac measure at~$0$, and 
$k \in L^1_{\textup{loc}}(0,T)$ is a completely monotone function.

To handle the vector valued case, we introduce the notation of a Hadamard product of $a,b \in \R^m$, defined by  
\[
a \odot b = \sum_{i=1}^m (a_i b_i) e_i,
\]
that is, component wise multiplication. With this notation, the component wise convolution of a function $y \in L^1(0,T)$ with a kernel $\mathcal{J} \in L^1(0,T)$ can be written as
\[
(\mathcal{J}*y)(t) = \int_0^t \mathcal{J}(t-s) \odot y(s) \ds =  \int_0^t \mathcal{J}(s) \odot y(t-s) \ds.
\]
As we use it frequently, note that by the Leipnitz integral rule,
\[
\ddt (\mathcal{J}*y)(t) = \ddt \int_0^t \mathcal{J}(s) \odot y(t-s) \ds = (\mathcal{J} *y_t)(t) + \mathcal{J}(t) \odot y(0) .
\]
We now state an auxiliary inequality in an abstract form.

\begin{lemma}\label{lemma:abstractkernel}
Let $T>0$, $d \in \N$, $\Omega \subseteq \R^d$ measurable, and let $p>1$. Assume $\mathcal{J} \in L^1(0,T)$
is completely monotone, and let $\mathfrak{r}$
be its resolvent of the first kind.
Furthermore, let $V\in W^{1,1}_{}\!\left(0,T;L^\infty(\Omega)^{}\right)$ be a positive definite diagonal matrix. Then, for a.e. $t \in (0,T)$ we have
\begin{equation} \label{eq:abstract_conv}
\begin{aligned}
2 \intt ( V(s) (\mathcal{J}* y)(s), y(s) )_{L^2(\Omega)} \ds 
\geq \,
& ( V(t) \mathbf{1} ,  \mathfrak{r} * (\mathcal{J} * y \odot \mathcal{J} * y )(t) )_{L^2(\Omega)} \\
- &\intt (V_t(s) \mathbf{1} , \mathfrak{r} * ( \mathcal{J} * y \odot \mathcal{J} * y )(s) )_{L^2(\Omega)} \ds
\end{aligned}
\end{equation}
for all $y\in L^2(0,t;L^2(\Omega))$.
\end{lemma}
\begin{proof}
We give the proof for $y\in C^\infty([0,T];L^2(\Omega))$, the general statement then follows by density of $C^\infty([0,T])$ in $L^2(0,T)$.

Since $\mathcal{J} \in L^1(0,T)$ is completely monotone in each component, we have that $\mathfrak{r} \in \mathcal{M}(0,T)$ is given by \eqref{resolvent_form} in each component, so that
\[ \label{eq:decompostion_measure}
\mathfrak{r}_i = a_i \delta_0 + k_i 
\]
with $a_i>0, k_i \in L^1(0,T)$ completely monotone. 
Note that complete monotonicity implies $\ddt k_i \leq 0$ and infinite differentiability on $(0,\infty)$ yielding that $k_i \in C^\infty(0,\infty) \subset W^{1,1}(\epsilon, T)$ for all $\epsilon>0$. 
Proceeding as in the proofs of \cite[Lemma B.1]{kaltenbacher2021determining} and \cite[Lemma 5.1]{kaltenbacher2024limiting}, we approximate $k_i$ by $$k_i^{(n)}(t) = \max \left\{ 1_{[0,\frac{1}{n}]} k_i(\frac{1}{n}) + 1_{[\frac{1}{n}, T]} k_i(t), \frac{\omega(t)}{n}  \right\}$$ 
with $\omega \in L^1(0,T)$ a fixed positive weight function and by performing the same calculation as in \cite[Lemma B.1]{kaltenbacher2021determining}, we obtain
the point wise in time estimate
\begin{equation} \label{eq:pointwisebound}
(k_i^{(n)} * (v_i)_t)(t) v_i(t) \geq \frac 12 (k_i^{(n)} * \dt (v_i^2))(t) 
\end{equation}  
for all $t \in (0,T)$ for some $v_i \in C^\infty([0,T]; L^2(\Omega))$ with $v_i(0)=0$. 
Since $V$ is diagonal and positive definite, multiplication of \eqref{eq:pointwisebound} by $V_{ii}(t)>0$ preserves the inequality. Summing over $i$ then yields
\begin{equation} \label{eq:vector_pointwsiebound}
V(t) (k^{(n)} * v_t)(t) \cdot v(t) \geq \frac 12 V(t) (k^{(n)} * \dt (v \odot v))(t) \cdot \mathbf{1} .
\end{equation}   
Integrating in \eqref{eq:vector_pointwsiebound} over $\Omega$ and in time on $(0,t)$, and using $v(0)=0$, we obtain, after integration by parts,
\begin{equation} \label{eq:abstractkernel1}
\begin{aligned}
&\intt ( V(s) (k * v_t)(s) ,  v(s) )_{L^2(\Omega)} \ds \geq \frac12 \intt ( V(s) \mathbf{1} , \dt (k *  (v \odot v))(s) )_{L^2(\Omega)} \ds \\
&= - \frac12 \intt ( V_t(s) \mathbf{1} , (k *  (v \odot v))(s) )_{L^2(\Omega)} \ds + \half ( V(t) \mathbf{1} , (k *  (v \odot v))(t) )_{L^2(\Omega)} .
\end{aligned}
\end{equation}
Here, we let $k^{(n)}$ converge to $k$ (by letting $n\to\infty$), since by the same arguments as in 
\cite[Lemma 5.1]{kaltenbacher2024limiting}, we have that $k_i^{(n)}$ converges to $k_i$ in $L^1$(0,T).  Observe that, with $A=\diagmatrix(a_1, \dots, a_m)$, we can apply the integration by parts formula from Lemma~\ref{le:mat_ibp} to obtain
\begin{equation}  \label{eq:abstractkernel2}
\begin{aligned}
&\intt ( V(s) (A \delta_0 * v_t)(s) ,  v(s) )_{L^2(\Omega)} \ds =  \intt ( V(s) A  v_t(s) ,  v(s) )_{L^2(\Omega)} \ds \\
&= -  \half \intt ( V_t(s) A  v(s) ,  v(s) )_{L^2(\Omega)} \ds
+ \half ( V(t) A  v(t) ,  v(t) )_{L^2(\Omega)}
\end{aligned}
\end{equation}
Adding up \eqref{eq:abstractkernel1} and \eqref{eq:abstractkernel2} using the decomposition \eqref{eq:decompostion_measure}, we obtain 
\begin{equation} \label{eq:conv_v}
\begin{aligned}
\intt ( V(s) (\mathfrak{r} * v_t)(s) ,  v(s) )_{L^2(\Omega)} \ds 
&\geq 
- \frac12 \intt ( V_t(s) \mathbf{1} , (\mathfrak{r} *  (v \odot v))(s) )_{L^2(\Omega)} \ds 
\\
&+ \half ( V(t) \mathbf{1} , (\mathfrak{r} *  (v \odot v))(t) )_{L^2(\Omega)}.
\end{aligned}
\end{equation}
Inserting $v = \mathcal{J} * y$ into \eqref{eq:conv_v}, yields $(\mathcal{J} * y)(0) = 0$. Consequently, $\mathfrak{r}* v_t = y$, since
\[
\mathfrak{r}_i * (\mathcal{J}_i * y_i)_t = \mathfrak{r}_i * ( \mathcal{J}_i * (y_t)_i + \mathcal{J}y_i(0)) = 1 \conv (y_t)_i + 1 y_i(0) = y_i - y_i(0) + y_i(0) .
\]
Thus, inequality \eqref{eq:abstract_conv} follows by density.
\end{proof}
\subsection{Proof of Lemma \ref{lemma:kernel}}\label{sec:proof34}
We now apply Lemma~\ref{lemma:abstractkernel} to prove Lemma \ref{lemma:kernel}. Let the assumptions of Lemma \ref{lemma:kernel} hold. 
First, we estimate the second term of \eqref{eq:abstract_conv} using Young's inequality 
\[
\begin{aligned}
&\intt (V_t(s) \mathbf{1} , \mathfrak{r} * ( \mathcal{J} * y \odot \mathcal{J} * y )(s) )_{L^2(\Omega)} \ds \\
&\leq   \| V_t \|_{L^\infty(0,T;L^\infty(\Omega))} \intt ( \mathbf{1} , \mathfrak{r} * ( \mathcal{J} * y \odot \mathcal{J} * y )(s) )_{L^2(\Omega)} \ds \\
&\leq  \| V_t \|_{L^\infty(0,T;L^\infty(\Omega))} \int_\Omega  \mathbf{1} \cdot  \intt \mathfrak{r} * ( \mathcal{J} * y \odot \mathcal{J} * y )(s) ) \ds dx \\
&\leq \| V_t \|_{L^\infty(0,T;L^\infty(\Omega))} ( \int_\Omega  \| {k} \|_{{L^1}(0,t)}  \| \mathcal{J} * y \odot \mathcal{J} * y \|_{L^1(0,t)} dx + \int_0^t\| A^\half \mathcal{J} * y\|_{L^2(\Omega)}^2 \ds )  \\
&\leq  \| V_t \|_{L^\infty(0,T;L^\infty(\Omega))} ( \| {k} \|_{{L^1}(0,T)}  + A_{max} ) \int_0^t \|  (\mathcal{J} * y)(s)\|_{L^2(\Omega)}^2  \ds .
\end{aligned}
\]
Here $A_{max}$ denotes the maximal eigenvalue of $A$.
For the first term in \eqref{eq:abstract_conv}, we have
\[
\begin{aligned}
& ( V(t) \mathbf{1} ,  \mathfrak{r} * (\mathcal{J} * y \odot \mathcal{J} * y )(t) )_{L^2(\Omega)}  
\\
&= \int_0^t ( V(t) \mathbf{1} , {k}(t-s)  \odot (\mathcal{J} * y)(s) \odot (\mathcal{J} * y )(s) )_{L^2(\Omega)} \ds \\
&+  ( V(t) A (\mathcal{J} * y)(t) , (\mathcal{J} * y)(t)  )_{L^2(\Omega)}
\\
&\geq  \inf_{s \in [0,t]} k(s) V_0 \int_0^t\|  (\mathcal{J} * y)(s)\|^2_{L^2(\Omega)} \ds +  \frac12 V_0 A_{min}\|  (\mathcal{J} * y)(t)\|^2_{L^2(\Omega)} \\
&\geq  \inf_{s \in [0,T]} k(s) V_0 \int_0^t\|  (\mathcal{J} * y)(s)\|^2_{L^2(\Omega)} \ds
\end{aligned}
\]
with $V_0=\inf_{s \in (0,t) \, x \in \Omega} \lambda_{\min}(V(s,x)).$ Combining the estimates, we obtain \eqref{eq:kernel_assump} 
with 
\[
2 C_{\mathcal{J}} =  \inf_{s \in [0,T]} k(s) V_0 - \| V_t \|_{L^\infty(0,T;L^\infty(\Omega))} ( \| {k} \|_{{L^1}(0,T)}  + A_{max} )  .
\]
To ensure $C_\mathcal{J}>0$, one needs 
\[ \label{eq:C_J_positive}
\| V_t \|_{L^\infty(0,T;L^\infty(\Omega))} < \frac{\inf_{s \in [0,T]} k(s) }{\| {k} \|_{{L^1}(0,T)}  + A_{max}} V_0 .
\]
We note that $\inf_{s \in [0,T]} k(s) >0$, since $\mathcal{J}$ is non-constant as shown in the proof of \cite[Lemma 5.2]{kaltenbacher2024limiting}.

\subsection{Proof of Lemma~\ref{lemma:kernel2}}\label{sec:proof35}
We now apply Lemma~\ref{lemma:abstractkernel} to prove Lemma \ref{lemma:kernel2}. Let the assumptions of Lemma \ref{lemma:kernel2} hold. In particular, $k_i \in L^q(0,T)$ and $a_i=0$ in \eqref{eq:decompostion_measure}. 
First, we estimate the second term of \eqref{eq:abstract_conv} using Hölder's inequality with $\xi>1$ such that $\frac{1}{\xi} + \frac{1}{q} = 1$ and  Young's inequality 
\[
\begin{aligned}
&\intt (V_t(s) \mathbf{1} , \mathfrak{r} * ( \mathcal{J} * y \odot \mathcal{J} * y )(s) )_{L^2(\Omega)} \ds \\
&\leq   \| V_t \|_{L^\xi(0,T;L^\infty(\Omega))} \left( \intt ( \mathbf{1} , \mathfrak{r} * ( \mathcal{J} * y \odot \mathcal{J} * y )(s) )_{L^2(\Omega)}^q \ds \right)^{\frac{1}{q}}\\
&\leq  \| V_t \|_{L^\xi(0,T;L^\infty(\Omega))} \int_\Omega  \mathbf{1} \cdot  \left( \intt \left( \mathfrak{r} * ( \mathcal{J} * y \odot \mathcal{J} * y )(s) \right)^q \ds \right)^\frac{1}{q}dx \\
&\leq \| V_t \|_{L^\xi(0,T;L^\infty(\Omega))}  \int_\Omega  \| {k} \|_{{L^q}(0,t)}  \| \mathcal{J} * y \odot \mathcal{J} * y \|_{L^1(0,t)} dx   \\
&\leq  T^{\frac{1}{\xi}} \| V_t \|_{L^\infty(0,T;L^\infty(\Omega))} \| {k} \|_{{L^q}(0,T)}  \int_0^t \|  (\mathcal{J} * y)(s)\|_{L^2(\Omega)}^2  \ds .
\end{aligned}
\]
For the first term in \eqref{eq:abstract_conv}, we have
\[
\begin{aligned}
& ( V(t) \mathbf{1} ,  \mathfrak{r} * (\mathcal{J} * y \odot \mathcal{J} * y )(t) )_{L^2(\Omega)}  
\\
&= \int_0^t ( V(t) \mathbf{1} , {k}(t-s)  \odot (\mathcal{J} * y)(s) \odot (\mathcal{J} * y )(s) )_{L^2(\Omega)} \ds 
\\
&\geq  \inf_{s \in [0,t]} k(s) V_0 \int_0^t\|  (\mathcal{J} * y)(s)\|^2_{L^2(\Omega)} \ds \\
&\geq  \inf_{s \in [0,T]} k(s) V_0 \int_0^t\|  (\mathcal{J} * y)(s)\|^2_{L^2(\Omega)} \ds
\end{aligned}
\]
with $V_0=\inf_{s \in (0,t) \, x \in \Omega} \lambda_{\min}(V(s,x)).$ Combining the estimates, we obtain \eqref{eq:kernel_assump} 
with 
\[
2 C_{\mathcal{J}} =  \inf_{s \in [0,T]} k(s) V_0 - \| V_t \|_{L^\infty(0,T;L^\infty(\Omega))}  T^{\frac{1}{\xi}}\| {k} \|_{{L^q}(0,T)}   .
\]
To ensure $C_\mathcal{J}>0$, we obtain the restriction \[ \label{eq:T_restriction}
T < \left( \frac{ \inf_{s \in [0,T]} k(s) V_0 }{ \| V_t \|_{L^\infty(0,T;L^\infty(\Omega))} \| {k} \|_{{L^q}(0,T)} } \right)^{\xi}.
\]
\subsection{Coervicity for non-diagonal \(V\) in the case of exponential kernels}\label{se:appendix_kernel_exponential_proof}
 Although \eqref{eq:vector_pointwsiebound} might fail when $V$ is not diagonal (see Section \ref{eq:non-diagonalcoefficient} below), one can still establish \eqref{eq:kernel_assump} for exponential kernels. Indeed, let $\mathcal{J}(t)=(e^{- a_1 t}, \dots, e^{- a_m t})^\transpose$ be a componentwise exponential kernel with $a_1, \dots, a_m > 0$. Then, by Leibniz integral rule
\[
\ddt (\mathcal{J} * y)(t) 
= \ddt  \int_0^t \mathcal{J}(t-s) y(s) \ds  
=  (\mathcal{J}_t * y)(t) + \mathcal{J}(0) y(t)
= - A (\mathcal{J} * y)(t) + 1 y(t)
\]
with $A=\diagmatrix(a_1, \dots, a_m)$. With $v = \mathcal{J} * y$, we obtain the identity 
\[ \label{eq:exponkernelid}
y = \ddt v + A v
\]
and we note that $v(0)=0$.
Inserting \eqref{eq:exponkernelid} into the left hand side of \eqref{eq:kernel_assump} we obtain
\[
\begin{aligned}
\int_0^t \int_\Omega V(s) &(\mathcal{J} * y)(s) \cdot y(s) dx \ds  
=\int_0^t \int_\Omega V(s) v(s) \cdot \left( \ddt v(s) + A v(s) \right) dx \ds = I_1 + I_2.
\end{aligned}
\]
Using partial integration Lemma \ref{le:mat_ibp} and symmetry of $V$ gives
\[
\begin{aligned}
I_1 &= \int_0^t \int_\Omega V(s) v(s) \cdot \dt v(s) dx \ds 
= \half \int_0^t \int_\Omega \dt \left( v(s)^\transpose V(s) v(s) \right) -   v(s)^\transpose  \dt V  v(s)  dx \ds \\
&\geq \half V_0 \| v(t) \|^2_{L^2(\Omega)}  - \half \| \dt V \|_{L^\infty(0,t;L^\infty(\Omega))} \int_0^t \| v(s) \|_{L^2(\Omega)}^2 \ds .
\end{aligned}
\]
For $I_2$, we have
\[
\begin{aligned}
I_2 =    \int_0^t \int_\Omega V(s) v(s) \cdot A v(s) dx \ds 
&=  \int_0^t \int_\Omega v(s)^\transpose V(s)  A  v(s) dx \ds 
\\&\geq  V_0 \min_{i}{a_i} \int_0^t \| v(s)\|^2_{L^2(\Omega)} \ds .
\end{aligned}
\]
Hence, in total we have
\[
\begin{aligned}
\int_0^t \int_\Omega V(s)& v(s) \cdot y(s) dx \ds 
\geq  
\left(V_0 \min_i{a_i} - \half \| \dt V \|_{L^\infty(0,t;L^\infty(\Omega))} \right) \int_0^t \| v(s) \|^2_{L^2(\Omega)} \ds .
\end{aligned}
\]
Note that this argument remains valid if $\mathcal{J}$ is a finite sum of exponential kernels.
%
%
%
%
%
\subsection{Remarks about certain aspects of the optimality of Lemma~\ref{lemma:kernel}}\label{se:appendix_lemma_assumptions}
We finish this Appendix by showing smallness assumption of the time derivative of the matrix \(V\)~\eqref{eq:kerlem_assump} is essential. Furthermore, we show that \eqref{eq:vector_pointwsiebound}, which is, for general kernels, an essential inequality for the proof of Lemma~\ref{lemma:kernel}, can fail for non-diagonal matrices \(V\). 
\subsubsection{Counter example for coefficient with large derivative}
The assumption that the time derivative of $V$ has to be small compared to a given constant, e.g., \eqref{eq:kerlem_assump} is necessary for \eqref{eq:kernel_assump} to hold, if $T>0$ is not chosen small enough.
To illustrate this choose 
$$\mathcal{J}(t) = e^{-t}, \quad  y(t)= 1- c t, \quad V(t) =c t + \half > 0$$
and $T=1$ with $c \in \R$. In this case we obtain $k=a=1$ as decomposition \eqref{resolvent_form} and \eqref{eq:C_J_positive} becomes
\[
| c | < \frac{e^{-1}}{1+1} \half = \frac{1}{4e} .
\]
Thus, choosing $c$ large enough we expect that inequality \eqref{eq:kernel_assump} fails. Indeed, setting $c=6$ we have 
$$
(\mathcal{J}*y)(t) = \int_0^t e^{-(t-s)} (1- 6 s) \ds = e^{-t}\int_0^t e^{s}  (1- 6 s) \ds =-7 {e}^{-t} - 6t + 7
$$
leading to 
\[ \label{eq:remark_kernel}
\begin{aligned}
\int_0^t V(s) (J*y)(s) y(s) \ds = \int_0^t (6 s + \half ) (-7e^{-s} - 6 s + 7) (1- 6 s) \ds
=
\\
\half {e}^{-t} \left(\left(108t^{4} - 180t^{3} + 18t^{2} + 7t + 959\right) {e}^{t} - 504t^{2} - 966t - 959\right)	.\end{aligned} 
\]
We show that there exists a $t$ such that \eqref{eq:remark_kernel} is negative. Since $\half e^{-t}$ is always positive, we can drop it from \eqref{eq:remark_kernel}. The expression being negative is thus equivalent to finding a $t$ such that
\[
e^{t} <  r(t) = \frac{504t^2 + 966t + 959}{108 t^4 - 180 t^2 + 7 t + 959} .
\]
At $t=\frac{3}{10}$ we have
\[
e^{\frac{3}{10}} < \frac{3235400}{2396837} = r(\frac{3}{10}) ,
\]
thus, due to the smoothness of the involved functions, \eqref{eq:kernel_assump} does not hold for a.e. $t \in (0,T)$ with a positive constant $C_\mathcal{J}.$
\subsubsection{Counter example of \eqref{eq:vector_pointwsiebound} for non-diagonal coefficient} \label{eq:non-diagonalcoefficient}
Even for positive definite $V$, \eqref{eq:vector_pointwsiebound} can fail if $V$ is not diagonal. A counterexample is
$$
k(t)= \begin{pmatrix}
e^{-t} \\  e^{-t} 
\end{pmatrix}
\quad 
V= \begin{pmatrix}
2 & -1 \\ -1 & 2
\end{pmatrix} \quad 
v(t)= \begin{pmatrix}
1-t \\  0 
\end{pmatrix} .
$$
Then, on the one hand, we have
\[
V v(t) \cdot (k*v_t)(t) =( 2 - 2t ) \int_0^t e^{-(t-s)} (- 1 )\ds = 2 ( 1 - t) (e^{-t} - 1)
\]
and on the other hand,
\[
\half V \mathbf{1} \cdot ( k * \dt (v \odot v ) )(t) =  \int_0^t e^{-(t-s)} (s-1) \ds = t - 2 + 2e^{-t}.
\]
Thus, inequality \eqref{eq:vector_pointwsiebound} holds if
$2 ( 1 - t) (e^{-t} - 1) \geq  t - 2 + 2e^{-t},$
which simplifies to
$- 2 t e^{-t}  +  t \geq 0 $, and ultimately gives $t \geq \log(2)$. Therefore,
the inequality fails for $t \in (0, \log(2)).$ 
We note that this does not mean that Lemma \ref{lemma:kernel} never holds for non-diagonal coefficient matrices, but it shows that it cannot be proven with the use of \eqref{eq:vector_pointwsiebound}.
\end{document}